\documentclass[12pt]{article}
\usepackage{graphicx} % Required for inserting images
\usepackage{float}    % option [H] for hard positioning of figures
\usepackage[T1]{fontenc}
\usepackage[utf8]{inputenc}
\usepackage[english]{babel}
\usepackage{tempora}  % Times-like font
\usepackage{amsmath}
\usepackage{amssymb}
\usepackage{amsthm}
\usepackage{amsfonts}
\usepackage{geometry}
\usepackage{thmtools}
\usepackage{dsfont}
\usepackage{setspace}
\usepackage[unicode, pdftex, hidelinks]{hyperref}

\newtheorem{lemma}{Lemma}
\newtheorem*{lemma_nnum}{Lemma}

\newtheorem{theorem}{Theorem}

\newtheorem{corollary}{Corollary}

\newtheorem{statement}{Proposition}

\newtheorem{hypothesis}{Conjecture}

\newtheorem*{note_nnum}{Remark}

\newtheorem{definition}{Definition}

\newcommand{\E}{\mathbb{E}}
\newcommand{\R}{\mathbb{R}}
\newcommand{\N}{\mathbb{N}}
\newcommand{\D}{\mathbb{D}}
\renewcommand{\P}{\mathbb{P}}
\renewcommand{\S}{\mathbb{S}}

\newcommand{\limto}[2][\infty]{\lim\limits_{#2\to #1}}

\newcommand{\distr}[1]{\mathcal{P}_{#1}}

\newcommand{\Ind}{\mathds{1}}

\newcommand{\so}{\Rightarrow}

\newcommand{\propP}{\texorpdfstring{\ensuremath{\mathrm{(P)}}}{(P)}}

\begin{document}
\thispagestyle{empty}
\begin{center}
  \hfill \break

  Saint Petersburg State University\\

  \hfill \break
  \hfill \break
  \hfill \break

  \large\textbf{Alexey Sergeevich Lotnikov}\\
  \large\textbf{Graduation Thesis}\\
  \begin{center}
  \large\bf{Mean distance between points inside and on the boundary of a convex body}
  \end{center}

\begin{center}
\footnotesize{Level of education: Bachelor's programme}\\
\footnotesize{Field of study 02.03.01 ``Mathematics and Computer Science''}\\
\footnotesize{Main educational programme SV.5189.2022 ``Data Science''}
\end{center}
   \hfill \break

\end{center}
\begin{flushright}
  \textbf{Academic advisor:}\\
  Leading Researcher,\\
  Laboratory ``Probabilistic Methods in Analysis'',\\
  Doctor of Physical and Mathematical Sciences,\\
  Corresponding Member of the Russian Academy of Sciences,\\
  Dmitry Nikolaevich Zaporozhets\\
  \hfill \break
  \textbf{Reviewer:}\\
  Senior Researcher,\\
  Laboratory ``Probabilistic Methods in Analysis'',\\
  St.\,Petersburg Department of the Steklov Mathematical Institute\\
  of the Russian Academy of Sciences,\\
  Candidate of Physical and Mathematical Sciences,\\
  Dmitry Vladimirovich Rutsky\\
\end{flushright}
\vfill
\begin{center}
  Saint Petersburg\\
  2026
\end{center}
\clearpage

\thispagestyle{empty}
\tableofcontents \newpage

\section{Introduction}

Consider a convex body $K \subset \mathbb{R}^n$. Throughout, by this we mean a convex compact set with non-empty interior.

\vspace{0.5em}
% \textbf{1. Sylvester's problem.}
\noindent
In 1864 J.~Sylvester posed the following problem of geometric probability: let four points $X_1, X_2, X_3, X_4 \in K$ be chosen at random and independently inside a planar convex body $K$. What is the probability that their convex hull $\mathrm{conv}(X_1, X_2, X_3, X_4)$ turns out to be a triangle? Obviously, this probability depends on the shape of $K$.

\noindent
In 1918 W.~Blaschke showed \cite{Blaschke1918} that for any convex body $K \subset \mathbb{R}^2$ the estimates
\[
\frac{35}{12\pi^2} \leq
P\bigl(\mathrm{conv}(X_1, X_2, X_3, X_4)\ \text{is a triangle}\bigr)
\leq \frac{1}{3},
\]
hold, the left bound being attained on ellipses and the right one on triangles.
This problem is one of the first examples of the application of probabilistic methods in convex geometry and gave rise to the area now known as geometric probability.

This inequality can be restated in terms of the mean area of a random triangle:
\[
P\bigl(\mathrm{conv}(X_1, X_2, X_3, X_4)\ \text{is a triangle}\bigr)
= 4 \cdot \,\frac{\mathbb{E}\,\mathrm{area}\,\mathrm{conv}(X_1,X_2,X_3)}{\mathrm{area}(K)}.
\]
Consequently, Blaschke's results are equivalent to the inequality
\[
\frac{35}{48\pi^2} \leq
\frac{\mathbb{E}\,\mathrm{area}\,\mathrm{conv}(X_1,X_2,X_3)}{\mathrm{area}(K)}
\leq \frac{1}{12},
\]
which gives sharp bounds for the normalised mean area of a random triangle built on points uniformly distributed inside $K$.

\vspace{0.5em}
% \textbf{2. Passing to the case of two points.}
Approaches of this kind naturally lead to the consideration of functionals related to the mean distance between random points.
In particular, the mean distance between two random points inside a convex body has been studied. Let
\[
\Delta(K) = \E|X_1X_2|,
\]
where $X_1$, $X_2$ are two independent random points chosen uniformly inside the body $K$ with respect to the corresponding Lebesgue measure. By $\Delta^p(K)$ we shall mean the expectation of the $p$-th power of the corresponding distance.
In the paper \cite{Bonnet2021} (Bonnet, Gusakova, Thäle, Zaporozhets, 2021) it was proved that
\[
\frac{7}{60} < \frac{\Delta(K)}{|\sigma K|} < \frac{1}{6},
\]

where $|\sigma K|$ is the surface area of our body. Note that both bounds are strict but asymptotically attainable: the lower one for degenerating isosceles triangles, and the upper one for a sequence of rectangles tending to a segment.
It is worth noting that we have stated only the two-dimensional version of the result, although the paper presents a more general fact, valid in an arbitrary dimension.

\vspace{0.5em}
% \textbf{3. The mean distance on the boundary.}
An analogous functional can be defined for random points on the boundary of the body. Let
\[
\theta(K) = \mathbb{E}|Y_1Y_2|,
\]
where $Y_1, Y_2$ are independent points uniformly distributed on $\sigma  K$ according to the corresponding Lebesgue measure. The definition of $\theta^p(K)$ is introduced analogously to $\Delta^p(K)$.
In 2022 A.~S.~Tokmachev \cite{Tokmachev2022} showed that for planar convex bodies the inequality
\[
\frac{\theta(K)}{|\sigma K|} \leq \frac{2}{\pi^2}
\]
holds, with equality if and only if $K$ is a disc. Thus the disc maximises the mean distance between two random boundary points for a fixed perimeter.

\vspace{0.5em}
% \textbf{4. The conjecture of Zaporozhets and Tarasov.}
In 2019 D.\,N.~Zaporozhets and A.\,S.~Tarasov put forward a natural conjecture comparing the interior and the boundary functionals:

\begin{hypothesis}\label{hyp:zt}
For any convex body $K$
\[
\Delta(K) \le \theta(K).
\]
\end{hypothesis}

\noindent
The conjecture can also be stated for all moments:

\begin{hypothesis}\label{hyp:zt-moments}
For any convex body $K$
\[
\Delta^{p}(K) \le \theta^{p}(K), \qquad p \geq 1.
\]
\end{hypothesis}

\noindent
\vspace{0.5em}
One may observe that in order to verify Conjecture~\ref{hyp:zt} it suffices to establish the following fact.

\begin{hypothesis}\label{hyp:theta-bound}
For any convex body $K$
\[
\cfrac{1}{6} \leq \cfrac{\theta(K)}{|\sigma K|}.
\]
\end{hypothesis}

\noindent
However, this statement also remains a conjecture.

\vspace{0.5em}
% \textbf{5. The current state of the question.}
Despite the simplicity of its statement, Conjecture~\ref{hyp:zt} has remained open for several years and has attracted considerable interest.
Its study was the main subject of the present work. We shall consider in detail the case of centrally symmetric bodies, and also point out some explicit relations that can be obtained for particular classes of bodies.
Among the results of this work, the centrally symmetric case is published in~\cite{lotnikov2025}, and the counterexample in higher dimensions is being prepared for publication as a preprint.

\section{Basic definitions}

Let us introduce several general notations that we shall use throughout the work. By the symbols $X$, $Y$, $Z, \ldots$ we shall denote random points. A prime will denote an independent copy of a point already introduced: thus, $X'$ is an independent copy of the point $X$. Whenever we speak of random points chosen inside or on the boundary of a body, $X$ will denote a random interior point and $Y$ a random point on the boundary (in the sense of the definitions given in the introduction). Of course, for the sake of rigour such notations should be indexed by the body inside which we choose the points. However, in most cases the body $K$ is clear from the context; where this is not so, we introduce additional indices.

It will also be convenient to use the term \emph{mean internal distance of a distribution}. By this we shall mean the mean distance between independent random variables drawn from the given distribution. Moreover, in this work we shall have to project points onto various directions a great deal, so we introduce the notation in advance: by $X_u$ we shall mean the projection of the point $X$ onto the direction $u$.

\section{Structure of the work}

The present work can be divided into several parts. After stating the results we shall pass to the consideration of centrally symmetric bodies and discuss in detail the differences between the one-dimensional and the multi-dimensional results. Next we shall pass to a discussion of the asymptotic behaviour of the quantities under study in the case of planar bodies, where we shall discuss the proof of a certain weakened version of Conjecture~\ref{hyp:zt}. The third part will be devoted to deriving exact relations between the functionals under study for some classes of bodies. In the last part we shall try to share various unexpected views on the problem and the results to which they lead, and we shall also show several counterexamples to extremely natural approaches to solving the presented problem.

\section{Results}
In this section we present the main results obtained in this work. We begin with the centrally symmetric case, where the question can be answered completely.

\subsection{The centrally symmetric case}

The Zaporozhets--Tarasov conjecture turns out to be true only in dimensions not exceeding two. This statement can be split into two theorems.

\begin{theorem}
    For any centrally symmetric body $K \subset \R^2$ with centre of symmetry $O$ and for any direction $u$ one has
    $$|P_u(OY)| \succ |P_u(OX)|.$$
\end{theorem}

\begin{corollary}
    $$\forall \, n > 0 \quad \E|OY|^n > \E|OX|^n.$$
\end{corollary}

\begin{corollary}
    $$\forall \, n \geq 1 \quad \E|YY'|^n \geq \E|XX'|^n.$$
\end{corollary}

The second statement shows that the conjecture fails in higher dimensions.
\begin{theorem}\label{thm:existence-d3}
    For any dimension $d \geq 3$ and any $p \geq 1$ there exists a centrally symmetric convex body $K \subset \R^d$ such that
    $$\Delta^p(K) > \theta^p(K).$$
\end{theorem}

In the section devoted to the centrally symmetric case we shall explicitly construct a family of such bodies and indicate the asymptotic behaviour of the defect.
\noindent
Thus, in the case of centrally symmetric bodies the answer to the Zaporozhets--Tarasov conjecture is obtained in explicit form.

\subsection{The general case of higher-order moments}
Let $K \subset \R^2$ be some convex body. In this case several results can be stated.

\begin{theorem}

\begin{enumerate}
    \item If
    $$\cfrac{|K|}{|\sigma K|^2} \geq \cfrac{1}{\sqrt{2(n+2)(n+3)}},$$
    then
     $$\E|XX'|^n \leq \E|YY'|^n.$$

     \item If
    $$\cfrac{|K|}{|\sigma K|^2} \geq \cfrac{1}{n+3},$$
    then
     $$\E|XX'|^n \leq \E|XY|^n.$$

     \item If
    $$\cfrac{|K|}{|\sigma K|^2} \geq \cfrac{1}{2(n+2)},$$
    then
     $$\E|XY|^n \leq \E|YY'|^n.$$
\end{enumerate}

\end{theorem}

\begin{corollary}
    For any convex body $K \subset \R^2$ there exists $C_K > 0$ such that
    $$\forall \, n > C_K \quad \E|XX'|^n \leq \E|XY|^n \leq \E|YY'|^n.$$
\end{corollary}

That is, at least for moments of high order we can guarantee the inequalities under study.

\subsection{The case of a circumscribed polygon}

In some cases, using the geometric features of the body, one can derive an exact correspondence between the mean distances under study. However, the following distribution turns out to be more natural for study from this point of view.

\begin{definition}[Projective distribution]
    Consider the distribution of a random point $Y_O$ on the boundary of a polygon obtained by projecting a random point $X$ onto the boundary of the body with centre at some fixed interior point $O$.
    The resulting distribution will be called projective with centre of projection at the point $O$.
\end{definition}

\begin{theorem}
    Let $K \subset \R^d$ be an arbitrary convex body, and let $O \in K$ be some fixed interior point. Then:
    \begin{enumerate}
        % \item The distribution of the distance $|XY_O|$ does not depend on the choice of the point $O$
        \item For any $n \geq 1$:
        \[
        \E|XX'|^n = \frac{2d}{n + 2d} \cdot \E|XY_O|^n.
        \]

        \item Characteristic functions:
        \[
        \E\exp(it \cdot |XY_O|) = \frac{\frac{d}{dt}\bigl[t^{2d} \cdot \E\exp(it \cdot |XX'|)\bigr]}{2d \cdot t^{2d-1}}.
        \]

        \item Distances to the centre:
        \[
        \E|OX|^n = \frac{d}{n + d} \cdot \E|OY_O|^n,
        \]
        \[
        \E\exp(it \cdot |OY_O|) = \frac{\frac{d}{dt}\bigl[t^{d} \cdot \E\exp(it \cdot |XO|)\bigr]}{d \cdot t^{d-1}}.
        \]
    \end{enumerate}
\end{theorem}

\begin{statement}
    For any circumscribed body $K \subset \R^d$ with centre of the inscribed sphere $O$ one has
    $$Y_O \overset{d}{=} Y.$$
\end{statement}

\begin{statement}
    For any centrally symmetric body $K \subset \R^d$ with centre of symmetry $O$ one has
    $$\E|XX'|^2 = \cfrac{d}{d+1} \cdot \E|XY_O|^2 = \cfrac{2d}{d+2} \cdot \E|OY_O|^2 = \cfrac{d}{d+2} \cdot \E|Y_OY_O'|^2.$$
\end{statement}

\begin{corollary}
    For any circumscribed centrally symmetric polytope $K \subset \R^d$ one has
    $$\E|XX'|^2 = \cfrac{d}{d+1} \cdot \E|XY|^2 = \cfrac{2d}{d+2} \cdot \E|OY|^2 = \cfrac{d}{d+2} \cdot \E|YY'|^2.$$
\end{corollary}

This equality constitutes an exact correspondence between the interior and the boundary moments of the second order.

\subsection{The case of a triangle}

The triangle is the simplest non-symmetric example, and for it the conjecture can not only be confirmed but also strengthened: an inequality holds between the one-dimensional projections of the distributions of the interior and the boundary points onto each direction separately. By the Cauchy--Crofton formula, such a per-projection inequality automatically implies the comparison of the full mean distances as well.

\begin{theorem}
    For any triangle $K \subset \R^2$ and any direction $u$ one has
    $$\E|P_u(X) - P_u(X')| \;\leq\; \E|P_u(Y) - P_u(Y')|.$$
    In particular, $\Delta(K) \leq \theta(K)$.
\end{theorem}

Moreover, the gap is at least a constant times the perimeter of the body: for any triangle $K$
$$\theta(K) - \Delta(K) \;\geq\; \tfrac{7}{240}\,|\sigma K|.$$

\section{Continuity of the functionals}
\label{sec:continuity}

The proofs of all the facts that we are going to present are sometimes much easier to carry out for polygons, and sometimes even for bodies with a smooth boundary. As is well known, any convex body can be approximated from inside by bodies of these classes in the Hausdorff metric. We only need to obtain the continuity of all the functionals in this metric.

In his work A.\,S.~Tokmachev~\cite{Tokmachev2022} showed that the functional $\theta^p$ is continuous in the Hausdorff metric. The continuity of the functionals $|K|$ and $|\sigma K|$ we shall regard as known. It only remains to understand what happens with the functional $\Delta^p$.

As we have already noted in the introduction, this functional is not continuous in the general sense, but a weaker statement does hold:

\begin{statement}
Let $f \colon \R \to \R$ be a measurable function, bounded on every segment $[0, C]$, and let $K_i \subset K$ for $i = 1, 2, 3, \ldots$ be a sequence of convex bodies converging in the Hausdorff metric to a body $K$ such that $|K| > 0$. Then
$$\limto{i} \E_{K_i} f(|XX'|) = \E_K f(|XX'|).$$
\end{statement}

\begin{proof}
Let $A_i = K_i \cap K$, $B_i = K \triangle K_i$. Also denote by $X$ a random point taken according to the uniform distribution inside $K$. By $X_i$ we shall mean a point uniformly distributed in $A_i$. Finally, by $Z_i$ we shall mean a random point uniformly distributed inside $B_i$. All points are assumed independent. In this notation
$$\E f(|XX'|) = \cfrac{|A_i|^2}{|K|^2} \cdot \E f(|X_iX_i'|) + \cfrac{2|A_i||B_i|}{|K|^2} \cdot \E f(|X_iZ_i|) + \cfrac{|B_i|^2}{|K|^2} \cdot \E f(|Z_iZ_i'|).$$

From the condition of convergence in the Hausdorff metric it follows that
$$\limto{i} |A_i| = |K|, \qquad \limto{i} |B_i| = 0,$$
and hence
$$\limto{i} \cfrac{|A_i|}{|K|} = 1, \qquad \limto{i} \cfrac{|B_i|}{|K|} = 0.$$

Since $f$ is bounded and the diameter of the body is finite, the second and the third summands tend to zero by Lebesgue's dominated convergence theorem. The distribution of the point $X_i$ converges weakly to the distribution of the point $X$, which together with the boundedness of the function $f$ entails the convergence of the expectations.
\end{proof}

In particular, we shall use this statement for the standard approximation of convex bodies by polygons from inside and for $f(x) = |x|^n$, which will allow us to carry out the proofs of the facts only for polygons.

\section{One-dimensional centrally symmetric distributions}
\label{sec:linear-symmetric}

Before passing to the proof of the two-dimensional and the multi-dimensional results, let us study the mean distances between two points taken independently from a general centrally symmetric distribution, since in what follows we shall try to reduce everything precisely to this case. Since on the whole only the distance between the points, or to the centre of symmetry, will matter to us, the latter may be taken to be zero without loss of generality, simply by applying a shift of the distribution.
\begin{definition}
Let a distribution $P$ on $\R$, symmetric with respect to $0$, be given. By its half we shall mean the distribution $P_{1/2}$ on $\R_+$ such that
    $$\forall A \subset \R_+ \quad P_{1/2}(A) = 2 \cdot P(A).$$

In essence, $P_{1/2}$ is the distribution of $|X|$, where $X \sim P$.
\end{definition}

\begin{definition}\label{def:wider}
We shall say that a centrally symmetric distribution $P$ is \textit{wider} than a centrally symmetric distribution $Q$ if $P_{1/2} \succ Q_{1/2}$.
\end{definition}
It is clear that a centrally symmetric distribution is essentially determined by its half, so it is natural to study the properties of $P$ precisely in terms of $P_{1/2}$. This is what the next proposition is about.
\begin{statement}
    Let $P$ be a distribution symmetric with respect to $0$, and let $f \colon \R \to \R$ be some function. Let $X$, $Y \sim P$ be independent, and let $X_{1/2}$, $Y_{1/2} \sim P_{1/2}$ be also independent of each other and of the variables already introduced. Then
    $$\E f(|X-Y|) = \E \left[\cfrac{f(|X_{1/2}-Y_{1/2}|) + f(|X_{1/2}+Y_{1/2}|)}{2}\right].$$
\end{statement}

\begin{proof}
    Let $S_1, S_2$ be binary variables, independent of each other and of the variables already introduced, equal to one with probability $\tfrac{1}{2}$ and to minus one with probability $\tfrac{1}{2}$. Note that $X \overset{d}{=} S_1 X_{1/2}$ and $Y \overset{d}{=} S_2 Y_{1/2}$. Then
    \begin{align*}
        \E f(|X-Y|) &= \E f(|S_1 X_{1/2} - S_2 Y_{1/2}|) \\
        &= \frac{1}{4}\,\E\Big[ f(|X_{1/2} - Y_{1/2}|) + f(|-X_{1/2} + Y_{1/2}|) \\
        &\qquad\quad {}+ f(|X_{1/2} + Y_{1/2}|) + f(|-X_{1/2} - Y_{1/2}|) \Big] \\
        &= \cfrac{\E f(|X_{1/2} - Y_{1/2}|) + \E f(|X_{1/2} + Y_{1/2}|)}{2}.
    \end{align*}
\end{proof}

Intuitively, a wider distribution ought to generate larger distances between independent points. The most natural manifestation of this fact would be the stochastic dominance of the distance between the first pair of points over the distance between the second; however, such a fact is obviously false in general. The intuition is realised in the following proposition.
\begin{statement}\label{stmt:wider-convex}
    Let $f \colon \R \to \R$ be a non-decreasing convex function (not necessarily strictly). Let $P, Q$ be two distributions symmetric with respect to $0$, with $P$ \textit{wider} than $Q$. Let $X, X' \sim P$ and $Y, Y' \sim Q$ be independent of each other. Then
    $$\E f(|X-X'|) \;\geq\; \E f(|X-Y|) \;\geq\; \E f(|Y-Y'|).$$
\end{statement}

\begin{proof}
    Let us look at $\E f(|X-X'|)$. By the proposition just proved,
    $$\E f(|X-X'|) = \E \left[\cfrac{f(|X_{1/2}-X_{1/2}'|) + f(|X_{1/2}+X_{1/2}'|)}{2}\right].$$

    \noindent
    Let us introduce the auxiliary function $g(x, y) = \cfrac{f(|x-y|) + f(|x+y|)}{2}$. Note that $g$ is non-decreasing in both variables on $\R_+ \times \R_+$. By symmetry it suffices to show monotonicity in $x$.

    \noindent
    First case: $x \geq y$. Then
    $$g(x,y) = \cfrac{f(x-y) + f(x+y)}{2}$$
    is non-decreasing in $x$, since $f$ is non-decreasing.

    \noindent
    Second case. Let $x_1 < x_2 \leq y$:
    $$\cfrac{f(y-x_1) + f(x_1+y)}{2} \leq \cfrac{f(y-x_2) + f(x_2+y)}{2}.$$
    Note that the sum of the arguments on the left-hand side and on the right-hand side equals $2y$, while their difference equals $2x_1$ and $2x_2$ respectively. It remains to use the fact that the sum of the values of a convex function does not increase when the arguments come closer together while their sum is preserved.

    \noindent
    Thus, it remains to show that there exists a non-decreasing function $F \colon \R_+ \to \R_+$ for which $X_{1/2} \overset{d}{=} F(Y_{1/2})$ and $F(y) \geq y$ for all $y \geq 0$. In this case, using the monotonicity of $g$ and the inequality $F(y) \geq y$,
    \begin{align*}
        \E f(|X-X'|) &= \E g(X_{1/2}, X'_{1/2}) = \E g(F(Y_{1/2}), F(Y'_{1/2})) \\
        &\geq \E g(F(Y_{1/2}), Y'_{1/2}) \geq \E g(Y_{1/2}, Y'_{1/2}) = \E f(|Y - Y'|).
    \end{align*}
    The middle term here equals $\E g(X_{1/2}, Y'_{1/2}) = \E f(|X - Y|)$ by the identity proved above, applied to the pair $X, Y'$; thereby the three-term chain is obtained as well.

    The required $F$ is supplied by Smirnov's transform: for the distribution functions $F_{X_{1/2}}, F_{Y_{1/2}}$ put $F = F^{-1}_{X_{1/2}} \circ F_{Y_{1/2}}$, so that $X_{1/2} \overset{d}{=} F(Y_{1/2})$. The function $F$ is non-decreasing as a composition of non-decreasing functions, and the inequality $F(y) \geq y$ is equivalent to the stochastic majorisation $F_{X_{1/2}} \leq F_{Y_{1/2}}$, that is, to the \textit{width} relation.
\end{proof}

Note that, in particular, the lemma can be applied to the functions $f(t) = t^n, n \geq 1$, which in the future will allow us to compare the moments of the distances between various random points.

\begin{statement}
    Let $P$, $Q$ be two symmetric distributions, with $P$ \textit{wider} than $Q$. Let $X, X' \sim P$, $Y, Y' \sim Q$ be independent of each other. Then:
    \begin{enumerate}
        \item $\E|X|^n \geq \E|Y|^n$ for all $n > 0$;
        \item $\E|X - X'|^n \geq \E|Y - Y'|^n$ for all $n \geq 1$.
    \end{enumerate}
\end{statement}
\begin{proof}

    The first statement follows directly from the definition of the \textit{width} relation.

    The second follows directly from the proposition above.
\end{proof}

For the construction of the counterexample in higher dimensions we shall need a \emph{strict} version of this chain. Let us introduce a suitable strengthening of the width relation.

\begin{definition}\label{def:strictly-wider}
We shall say that $P$ is \emph{strictly wider} than $Q$ if $F_{P_{1/2}}(t) < F_{Q_{1/2}}(t)$ for all $t$ for which $F_{Q_{1/2}}(t) \in (0, 1)$.
\end{definition}

\begin{statement}\label{stmt:strict-three-way}
Let $P$ be \emph{strictly wider} than $Q$, both distributions being continuous, and let $X, X' \sim P$, $Y, Y' \sim Q$ be independent. Then for all $n \geq 1$
$$\E|X-X'|^n \;>\; \E|X-Y|^n \;>\; \E|Y-Y'|^n.$$
\end{statement}

\begin{proof}
We keep the notation of the proof of Proposition~\ref{stmt:wider-convex} with $f(t) = t^n$: the auxiliary function $g(x, y) = \tfrac{1}{2}\bigl(|x-y|^n + (x+y)^n\bigr)$ and the non-decreasing $F$ with $X_{1/2} \overset{d}{=} F(Y_{1/2})$ and $F(y) \geq y$. Putting $B = Y_{1/2}$, $B' = Y'_{1/2}$, we have
$$\E|X-X'|^n = \E g(F(B), F(B')), \qquad \E|X-Y|^n = \E g(F(B), B'), \qquad \E|Y-Y'|^n = \E g(B, B').$$
The non-strict inequalities between these quantities have already been proved in Proposition~\ref{stmt:wider-convex}; let us establish strictness. Strict width strengthens $F(y) \geq y$ to $F(y) > y$ everywhere where $F_{Q_{1/2}}(y) \in (0, 1)$; by the continuity of $Q_{1/2}$ this yields $F(B) > B$ and $F(B') > B'$ almost surely.

\emph{The case $n \geq 2$.} The function $g$ is strictly increasing in each argument on $(0, \infty)$ (for $y < x$ the derivative in $y$ equals $\tfrac{n}{2}\bigl((x+y)^{n-1} - (x-y)^{n-1}\bigr) > 0$, and for $y > x$ it is positive all the more). Since $F(B') > B'$ and $F(B) > B$ almost surely, both inequalities are strict.

\emph{The case $n = 1$.} Here $g(x, y) = \max(x, y)$. The first inequality is strict on the event $\{B' > F(B)\}$: on it $\max\bigl(F(B), F(B')\bigr) = F(B') > B' = \max\bigl(F(B), B'\bigr)$. The second is strict on the event $\{B' < F(B)\}$: on it $\max\bigl(F(B), B'\bigr) = F(B) > \max(B, B')$. Both events have positive probability.
\end{proof}

\section{The centrally symmetric case}
\label{sec:centrally-symmetric}

In this section we shall study \textit{centrally symmetric planar} bodies. We shall try to answer most of the questions related to this case. First of all, as was announced earlier, let us give a proposition thanks to which we shall be able to reduce the two-dimensional case to the one-dimensional one.

\begin{statement}[Cauchy--Crofton formula]\label{stmt:cauchy-crofton}
Let $X, Y$ be some pair of random points having some (possibly dependent) joint distribution. Then
$$\E|XY|^n = c_n \int_0^\pi \E|P_u(XY)|^n\, du, \qquad c_n = \left[\int_0^\pi |\cos x|^n\, dx\right]^{-1}.$$
\end{statement}

\begin{proof}
The statement follows trivially from Fubini's theorem:
$$c_n \int_0^\pi \E|P_u(XY)|^n\, du = c_n \E \int_0^\pi |P_u(XY)|^n\, du = c_n \E \int_0^\pi |\cos u|^n \cdot |XY|^n\, du = \E|XY|^n.$$
\end{proof}

Hence, in order to study the moments it suffices to study only the moments of the distances between the projections, which is already much easier. The following lemma will help us with this.

\begin{lemma}
Let $K$ be a non-empty convex body, centrally symmetric with respect to a point $O$, and let some direction $\vec{u}$ be chosen. The lines $l_1$, $l_2$ are centrally symmetric to each other with respect to the point $O$ and perpendicular to $\vec{u}$. Call $H$ the part of the plane contained between the lines $l_1$ and $l_2$. Then
$$\cfrac{|\sigma K \cap H|}{|\sigma K|} \leq \cfrac{|K \cap H|}{|K|}.$$
\end{lemma}

\begin{proof}
    Note that it suffices to consider centrally symmetric polygons, since in this case it is not difficult to pass to the limit in the inequality in the Hausdorff metric (see Section~\ref{sec:continuity}).

    Call $H_-$ the half-plane adjacent to $l_1$ and not containing $O$, and $H_+$ the half-plane centrally symmetric to it with respect to $O$. Let $P_1 = |\sigma K \cap H_+|$, $P = |\sigma K \cap H|$, $S_1 = |K \cap H_+|$, $S = |K \cap H|$. Then $|K| = S + 2S_1$, $|\sigma K| = P + 2P_1$.

    \begin{figure}[h] % figure environment (one may use [h], [t], [b])
        \centering
        \includegraphics[width=0.6\textwidth]{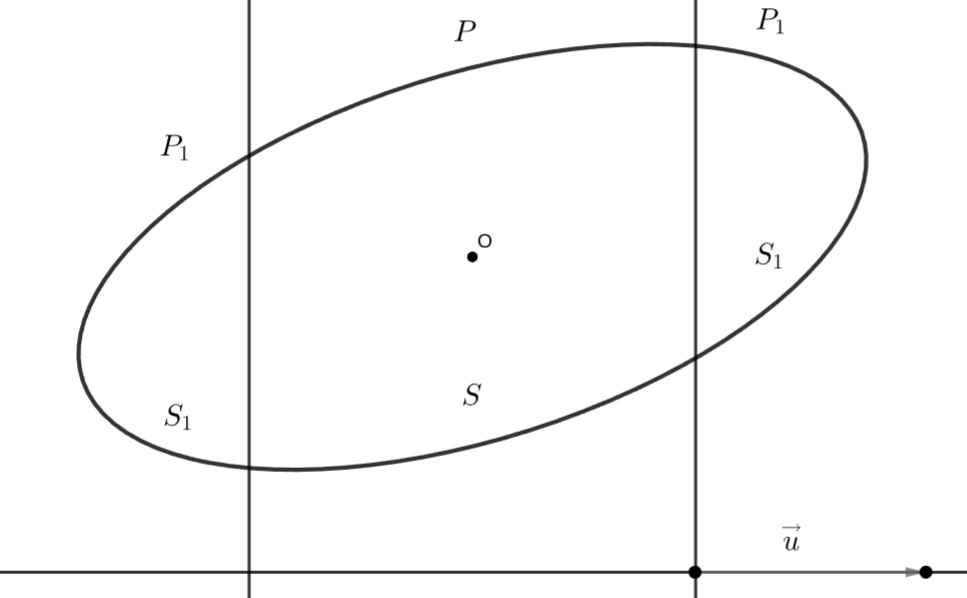} % insert the picture
        \label{fig:pcsb1} % label for references in the text
    \end{figure}

     It is obvious that the stated assertion is equivalent to the following one:
    $$\cfrac{S_1}{P_1} \leq \cfrac{|K|}{|\sigma K|}.$$

    We shall move the line $l_1$ from the position where it touches our body to the position where it passes through the centre of symmetry. We shall specify the line by the coordinate of its intersection with the line codirectional with $\vec{u}$; we introduce the coordinates so that $0$ corresponds to the moment of tangency and $1$ to the moment of passage through the centre of symmetry. We denote the line at the moment $t$ by $l_1(t)$. We define $S_1(t)$ and $P_1(t)$ analogously.
    Hence, it suffices for us to show that

    $$\forall \, 0 \leq t \leq 1 \quad \cfrac{S_1(t)}{P_1(t)} \leq \cfrac{S_1(1)}{P_1(1)} = \cfrac{|K|}{|\sigma K|}.$$

    That is, it suffices to verify the non-strict monotonicity of $R(t) = \cfrac{S_1(t)}{P_1(t)}$. The function $R(t)$ is piecewise differentiable: $S_1(t)$ is differentiable everywhere, and $P_1(t)$ everywhere except the moments when $l_1(t)$ passes through a vertex of the polygon. In other words, it suffices to verify the non-negativity of the derivative at all points where it exists:
    $$\cfrac{d}{dt}R(t) = \cfrac{S_1'(t)\, P_1(t) - P_1'(t)\, S_1(t)}{P_1(t)^2} \geq 0 \;\iff\; \cfrac{S_1(t)}{P_1(t)} \leq \cfrac{S_1'(t)}{P_1'(t)}.$$

    Note that $S_1'(t) = a(t)$ is the length of the segment of intersection of $K$ and $l_1(t)$. At the same time $P_1'(t) = \cfrac{1}{\sin(\alpha_1(t))} + \cfrac{1}{\sin(\alpha_2(t))}$, where $\alpha_1(t), \alpha_2(t)$ are the angles between the intersected side and $l_1(t)$.

    It only remains to verify the inequality
    $$\forall \, 0 \leq t \leq 1 \quad \cfrac{S_1(t)}{P_1(t)} \leq \cfrac{a(t)}{\cfrac{1}{\sin(\alpha_1(t))} + \cfrac{1}{\sin(\alpha_2(t))}}.$$

    It can be strengthened:
    $$\forall \, 0 \leq t \leq 1 \quad \cfrac{S_1(t)}{P_1(t)} \leq \cfrac{a(t) \cdot \min\bigl(\sin\alpha_1(t),\, \sin\alpha_2(t)\bigr)}{2}.$$

    To this end consider the strip formed by two parallel lines, and a segment in it. Suppose that in one of the resulting half-strips there fits a convex figure containing the whole segment of length $a$ on its boundary. Let its area be $\hat{S}$, its perimeter (without the length of the segment) $\hat{P}$, and the angle between the segment and the bases of the strips $\alpha$.

    Then

    $$\hat{S} \leq \cfrac{\hat{P} \cdot a \cdot \sin \alpha}{2}.$$

    In order to see this, let us construct a second arc, centrally symmetric to ours with respect to the midpoint of the segment. The resulting figure will have twice the area of the original one.
    The notation is as in the picture.

    \begin{figure}[h]
        \centering
        \includegraphics[width=0.8\textwidth]{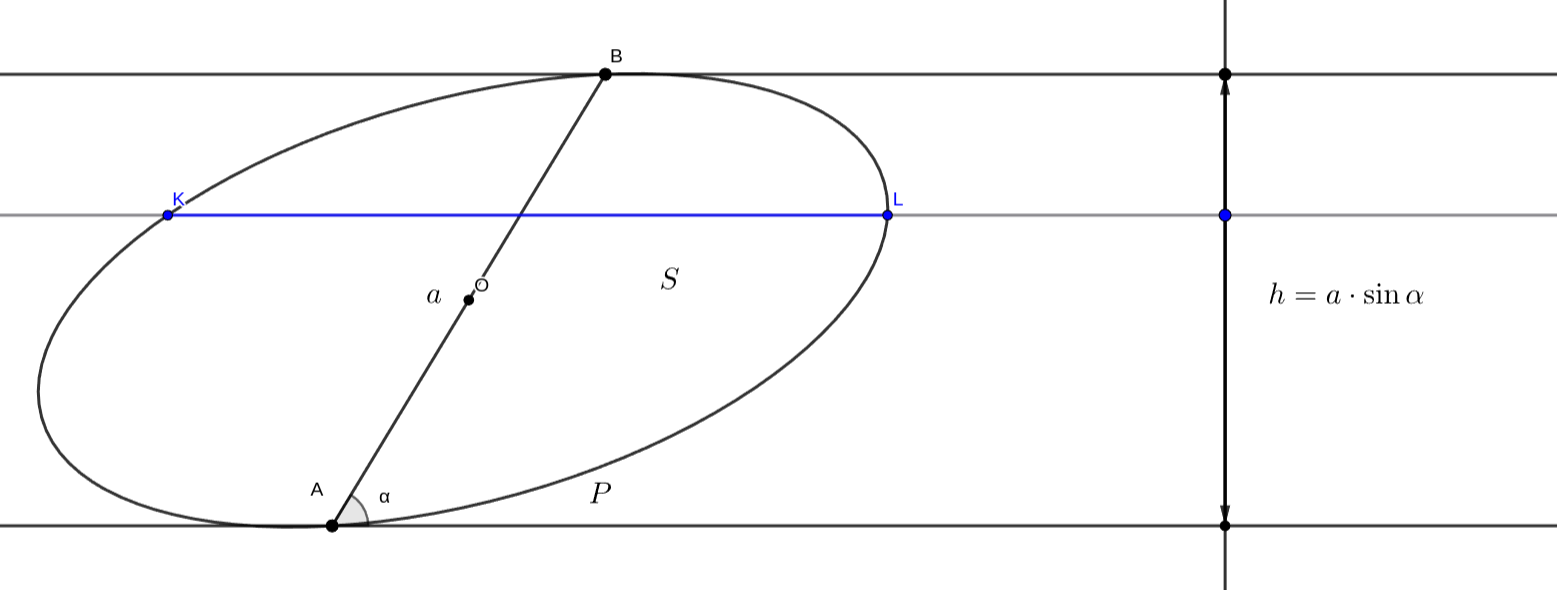}
        \label{fig:pcsb4}
    \end{figure}

    Let us represent its area as an integral:
    $$2S = \int_0^h |K \cap l(x)|\, dx \leq \int_0^h P\, dx = P h = P \cdot a \sin\alpha.$$

    Whence the assertion follows, since the length of any section cannot exceed half the length of the perimeter of the figure.

    From the geometric lemma we can obtain a multitude of useful probabilistic corollaries.
\end{proof}

\begin{corollary}\label{cor:planar-symmetric}
Let $K$ be a centrally symmetric convex body with centre at the point $O$, let $u$ be some direction, let $X$ be a random interior point of $K$, and let $Y$ be a random boundary point. Then
\begin{enumerate}
    \item $|P_u(OY)| \succ |P_u(OX)|$;
    \item $P_u(Y)$ is \textit{wider} than $P_u(X)$;
    \item $\E|OY|^n > \E|OX|^n$ for all $n > 0$;
    \item $\E|YY'|^n \geq \E|XX'|^n$ for all $n \geq 1$.
\end{enumerate}
\end{corollary}

\begin{proof}
    Statement~1 is by definition equivalent to the lemma described above, the second is derived from the first by the definition, and statements~3 and~4 are obtained from the second by applying the properties of the ``wider'' relation from Section~\ref{sec:linear-symmetric}.
\end{proof}

Thus, we have shown that Conjecture~\ref{hyp:zt} holds for planar centrally symmetric bodies.

\section{Counterexamples}

In the previous section we established that for planar centrally symmetric convex bodies Conjecture~\ref{hyp:zt} holds: the moments of the distance between two random boundary points majorise the moments of the distance between two random interior points. It is natural to ask whether this result carries over to dimensions $d \geq 3$.

The purpose of the present section is to show that it does not: in every dimension $d \geq 3$ Conjecture~\ref{hyp:zt} is refuted in all powers at once. Namely, for any $n \geq 1$ there is a centrally symmetric convex body $K \subset \R^d$ for which $\E|XX'|^n > \E|YY'|^n$.

The exposition is arranged in three steps. First a geometric condition on the body is introduced --- property~\propP{} --- from which, by a general ``stretching'' scheme, a counterexample in an arbitrary dimension is obtained. Then an explicit example of a body with property~\propP{} in $\R^3$ is given. Finally, the operation of raising the dimension transfers the counterexample to all $d \geq 3$.

\subsection{Construction of a counterexample from a body with property \propP{} in an arbitrary dimension}
\label{sec:counterexample-from-P}

Let $K_0 \subset \R^d$ be a centrally symmetric convex body with centre at the origin, and let $u \in S^{d-1}$ be a unit vector. The sections of the body by hyperplanes orthogonal to $u$,
\[
    K_0(t) \;=\; K_0 \cap \{x \in \R^d : \langle x, u \rangle = t\},
\]
are $(d-1)$-dimensional convex bodies in the hyperplane $\{\langle x, u \rangle = t\}$. Denote by $A_0(t) = \mathcal{H}^{d-1}(K_0(t))$ their $(d-1)$-dimensional volume, and by $P_0(t) = \mathcal{H}^{d-2}(\partial K_0(t))$ the $(d-2)$-dimensional surface measure of the relative boundary.

\begin{definition}[Property \propP]\label{def:propP}
A centrally symmetric convex body $K_0 \subset \R^d$ with centre at the origin has \emph{property \propP} if there is a direction $u \in S^{d-1}$ for which the function
\[
    \rho_{u}(t) \;=\; \frac{A_0(t)}{P_0(t)}
\]
is strictly increasing in $|t|$ on the support of $K_0$ relative to $u$.
\end{definition}

The quantity $\rho_u(t)$ is, up to a factor, the reciprocal of the \emph{Cheeger constant} of the section $K_0(t)$. For a bounded domain $\Omega$ the Cheeger constant
\[
    h(\Omega) \;=\; \inf_{F \subset \Omega} \frac{P(F)}{|F|}
\]
was introduced by Cheeger in 1970 \cite{cheeger70} in connection with the estimation of the first eigenvalue of the Laplace operator; the minimisers of this ratio are called \emph{Cheeger sets}. The ratio $A_0(t)/P_0(t)$ is exactly the ``volume/surface'' functional whose minimisation defines the Cheeger constant. The properties of its monotonicity in families of convex bodies have been studied in a number of works \cite{kl06, leonardi, altercaselles, cheegerpar}.

Before passing to the construction scheme itself, let us note that in low dimensions property \propP{} is unattainable.

\begin{note_nnum}\label{thm:planar-obstruction}
No centrally symmetric convex body in $\R^2$ has property \propP. Equivalently: for any such body and any direction $u \in S^1$ the function $\rho_u(t)$ is non-increasing in $|t|$.
\end{note_nnum}

The obstruction is elementary. For $d = 2$ the section $K_0(t)$ is a segment of length $L(t)$, where $L$ is an even and concave function, so $A_0(t) = L(t)$ is non-increasing in $|t|$. At the same time $P_0(t) \equiv 2$ is constant (the boundary of a segment consists of two points). Consequently, $\rho_u(t) = L(t)/2$ decreases in $|t|$ --- exactly the opposite of what is required by property \propP. In the one-dimensional case the situation is even simpler: the sections degenerate to a point, and both quantities $A_0$, $P_0$ cease to have a meaningful sense, so that property \propP{} is simply not defined.

Although Theorem~\ref{thm:main} stated below formally works in low dimensions as well, in $\R^1$ and $\R^2$ it is not applicable because of the absence of bodies with property~\propP, so in what follows we consider only $d \geq 3$.

Let us pass to the scheme itself. Fix $K_0 \subset \R^d$, $d \geq 3$, a centrally symmetric convex body with property \propP{} in the direction $u \in S^{d-1}$, with the functions $A_0(t)$, $P_0(t)$ from Definition~\ref{def:propP}. For $h > 0$ denote by $K_0(h)$ the image of $K_0$ under the anisotropic stretching multiplying the $u$-coordinate by $h$ and leaving the coordinates on the orthogonal complement $u^\perp$ unchanged. Let $X, X'$ be independent uniformly distributed points inside $K_0(h)$, and let $Y, Y'$ be independent uniformly distributed points on $\partial K_0(h)$.

With the body $K_0$ we associate two probability measures on $[-1, 1]$ with densities proportional to $A_0(|t|)$ and $P_0(|t|)$:
\[
    \mu_X(\mathrm{d}t) \;\propto\; A_0(|t|)\,\mathrm{d}t, \qquad
    \mu_Y(\mathrm{d}t) \;\propto\; P_0(|t|)\,\mathrm{d}t.
\]
They should be understood as the limiting, as $h \to \infty$, distributions of the projection onto the direction $u$, normalised by $h$, of an interior and of a boundary point of the body $K_0(h)$ respectively. The idea of the construction of the counterexample is that property \propP{} entails the width inequality: $\mu_X$ turns out to be \emph{wider} than $\mu_Y$ in the sense of Definition~\ref{def:wider}.

Let $\xi \sim \mu_X$ and $\eta \sim \mu_Y$ be independent random variables. For each $n \geq 1$ put
\[
    C^n_{XX} = \E|\xi - \xi'|^n, \qquad
    C^n_{XY} = \E|\xi - \eta|^n, \qquad
    C^n_{YY} = \E|\eta - \eta'|^n.
\]

\begin{theorem}[On the construction of a counterexample]\label{thm:main}
In the notation adopted:
\begin{enumerate}
    \item[\textup{(i)}]
        $C^n_{XX} > C^n_{XY} > C^n_{YY}$ for each $n \geq 1$;
    \item[\textup{(ii)}]
        as $h \to \infty$,
        \[
            \frac{\E|X - X'|^n}{h^n} \to C^n_{XX}, \qquad
            \frac{\E|X - Y|^n}{h^n} \to C^n_{XY}, \qquad
            \frac{\E|Y - Y'|^n}{h^n} \to C^n_{YY}.
        \]
\end{enumerate}
\end{theorem}

A direct consequence of the theorem is the existence of a counterexample for each power $n$.

\begin{corollary}\label{cor:counterexamples}
Under the hypotheses of Theorem~\ref{thm:main}, for any $n \geq 1$ there is $H_n > 0$ such that
\[
    \E|X - X'|^n \;>\; \E|X - Y|^n \;>\; \E|Y - Y'|^n
    \qquad \text{for all } h > H_n.
\]
In particular, for each $n \geq 1$ and $h$ large enough Conjecture~\ref{hyp:zt} fails on the family $K_0(h)$.
\end{corollary}

\begin{proof}
Each of the three ratios $\E|X-X'|^n/h^n$, $\E|X-Y|^n/h^n$, $\E|Y-Y'|^n/h^n$ converges to the corresponding limit by part~(ii) of Theorem~\ref{thm:main}, and the limits themselves are strictly ordered, $C^n_{XX} > C^n_{XY} > C^n_{YY}$, by part~(i). Therefore, starting from some $H_n$, this strict ordering is preserved for the quantities themselves as well.
\end{proof}

Let us pass to the proof of Theorem~\ref{thm:main}. The idea reproduces the scheme of Section~\ref{sec:centrally-symmetric}, run in the opposite direction: instead of the majorisation of the projections of the boundary point over the interior one, we shall show that under property~\propP, in the limit $h \to \infty$, the scaled projection of the \emph{interior} point turns out to be strictly wider than the projection of the \emph{boundary} one.

Fix $K_0$ as in the hypothesis of the theorem. Denote $\pi_u(x) = \langle x, u \rangle$, and for each $h > 0$ denote by $X_h, X_h'$ independent uniformly distributed interior points of the body $K_0(h)$, and by $Y_h, Y_h'$ independent uniformly distributed points on $\partial K_0(h)$.

\subsubsection{Limiting distributions of the projections}

The supports of $\pi_u(X_h)$ and $\pi_u(Y_h)$ grow linearly in $h$, so it is natural to work with the scaled projections
\[
    T^X_h := \frac{\pi_u(X_h)}{h}, \qquad T^Y_h := \frac{\pi_u(Y_h)}{h},
\]
both with support in $[-1, 1]$ and both symmetric with respect to zero.

\paragraph*{The interior projection.} The anisotropic stretching scales the $u$-coordinate by $h$ and leaves the coordinates on $u^\perp$ unchanged. Therefore the section of $K_0(h)$ at the level $\pi_u = s$ is isometric to the section $K_0(s/h)$ of the original body $K_0$ as a subset of $u^\perp$, and, in particular, their $(d-1)$-dimensional volumes coincide and equal $A_0(s/h)$. The density of $\pi_u(X_h)$ at a point $s \in [-h, h]$ equals
\[
    \frac{A_0(s/h)}{\mathrm{vol}_d(K_0(h))},
\]
where $\mathrm{vol}_d(K_0(h)) = h \cdot V_0$, $V_0 := \int_{-1}^1 A_0(t)\, dt$. After the change of scale, the density of $T^X_h$ has the form
\[
    f_{T^X_h}(t) = \frac{A_0(t)}{V_0}, \qquad t \in [-1, 1],
\]
and \emph{does not depend on $h$}.

\begin{statement}\label{prop:interior-proj}
For any $h > 0$ the distribution of the random variable $T^X_h$ coincides with the measure $\mu_X$ on $[-1, 1]$ having density $A_0(|t|)/V_0$. In particular, on the side of the interior point no passage to the limit is required.
\end{statement}

\paragraph*{The boundary projection.} On the boundary side the passage to the limit is already non-trivial. The idea is simple: under a strong stretching along $u$ the lateral boundary $\partial K_0(h)$ becomes almost parallel to $u$, so almost all of its $(d-1)$-dimensional measure is concentrated on the ``side of the cylinder'', and its cross-section gives exactly the measure $\mu_Y$.

Let us formalise this via Federer's coarea formula~\cite[Theorem~3.2.22]{federer}: for a hypersurface $M$ and a direction $u$ the density of the image $(\pi_u)_* \mathcal{H}^{d-1}|_M$ at a regular value $s$ equals
\[
    \rho(s) = \int_{M \cap \{\pi_u = s\}} \frac{d\mathcal{H}^{d-2}(y)}{\sin\theta(y)},
\]
where $\theta(y)$ is the angle between $u$ and the outer normal to $M$ at the point $y$. In order to write down with its help the density of $T^Y_h$, we need to know three things about $M = \partial K_0(h)$: the unit normal $\nu_h$, the angle $\theta_h$ between $\nu_h$ and $u$, and the total $(d-1)$-measure of $\partial K_0(h)$ (for the normalisation into a probability density).

\medskip
\noindent\emph{The normal and the angle.} The anisotropic stretching $M_h = \mathrm{diag}(h, I_{d-1})$ takes a point $x \in \partial K_0$ to $M_h(x) \in \partial K_0(h)$. The normal is transformed as follows: if $\nu(x) = (n_1, \mathbf{n}_2)$ is the unit normal to $\partial K_0$ at the point $x$ (with the decomposition $n_1$ along $u$, $\mathbf{n}_2 \in u^\perp$), then the normal to $\partial K_0(h)$ at the point $M_h(x)$ is proportional to $M_h^{-\top} \nu = (n_1/h, \mathbf{n}_2)$. Normalising, we obtain
\[
    \sin\theta_h(x) = \frac{\|\mathbf{n}_2\|}{\sqrt{n_1^2/h^2 + \|\mathbf{n}_2\|^2}}.
\]
In the limit $h \to \infty$ this gives $\sin\theta_h \to 1$ everywhere where $\|\mathbf{n}_2\| > 0$ --- that is, everywhere except the ``horizontal caps'' of $\partial K_0$, whose normal is parallel to $u$. This is precisely the quantity appearing in the denominator of the coarea formula.

\medskip
\noindent\emph{The total $(d-1)$-measure.} For a linear map $M$ the Jacobian carrying the $(d-1)$-dimensional measure from a hypersurface with unit normal $\nu$ to its image under $M$ equals $|\det M| \cdot \|M^{-\top} \nu\|$. Applying this to $M = M_h$ and $\partial K_0$:
\[
    \mathcal{H}^{d-1}(\partial K_0(h))
    = \int_{\partial K_0} h \cdot \|(n_1/h, \mathbf{n}_2)\|\, d\mathcal{H}^{d-1}
    = \int_{\partial K_0} \sqrt{n_1^2 + h^2 \|\mathbf{n}_2\|^2}\, d\mathcal{H}^{d-1}.
\]
Since $h\|\mathbf{n}_2\| \leq \sqrt{n_1^2 + h^2 \|\mathbf{n}_2\|^2} \leq h\|\mathbf{n}_2\| + |n_1|$, we have
\[
    \mathcal{H}^{d-1}(\partial K_0(h))
    = h \int_{\partial K_0} \|\mathbf{n}_2\|\, d\mathcal{H}^{d-1} + O(1)
    = h S_0 + O(1),
\]
where $S_0 := \int_{\partial K_0} \|\mathbf{n}_2\|\, d\mathcal{H}^{d-1} = \int_{-1}^{1} P_0(t)\, dt$ is the integral of $\|\mathbf{n}_2\|$ over $\partial K_0$ (the same thing as the coarea formula for $\pi_u$ on $\partial K_0$ itself with the test function $g \equiv 1$: the integrand there is $|\nabla^M \pi_u| = \sin\theta_0 = \|\mathbf{n}_2\|$).

\medskip
\noindent\emph{The slice.} Since $M_h$ is the identity on $u^\perp$, the slice $\partial K_0(h) \cap \{\pi_u = ht\}$ is isometric to the slice $\partial K_0 \cap \{\pi_u = t\}$ of the original body, so its $(d-2)$-measure equals $P_0(t)$.

\medskip
\noindent\emph{The density.} Substituting everything into the coarea formula and passing to the normalised projection $T^Y_h = \pi_u(Y_h)/h$,
\[
    f_{T^Y_h}(t)
    = \frac{h}{\mathcal{H}^{d-1}(\partial K_0(h))}
      \int_{\partial K_0(t)} \frac{d\mathcal{H}^{d-2}(y)}{\sin\theta_h(t, y)}
    = \frac{h}{h S_0 + O(1)}
      \int_{\partial K_0(t)} \frac{d\mathcal{H}^{d-2}(y)}{\sin\theta_h(t, y)}.
\]
This is the \emph{explicit density} of the projection $T^Y_h$ for each finite $h$.

\medskip
\noindent\emph{Passage to the limit.} Fix $t \in (-1, 1)$ and a compact set $[-1+\varepsilon, 1-\varepsilon]$ around it. On the slice $\partial K_0(t)$ the quantity $\|\mathbf{n}_2\|$ is bounded away from zero (the normals point strictly sideways rather than along $u$), so there is $c > 0$ such that $\sin\theta_h(t, y) \geq c$ for all $y$ on the slice and all $h \geq 1$. The integrand $1/\sin\theta_h(t, y)$ is bounded by the constant $1/c$ and converges pointwise to $1$. By Lebesgue's dominated convergence theorem,
\[
    \int_{\partial K_0(t)} \frac{d\mathcal{H}^{d-2}(y)}{\sin\theta_h(t, y)}
    \;\xrightarrow[h \to \infty]{}\;
    P_0(t).
\]
Combining this with $\mathcal{H}^{d-1}(\partial K_0(h)) = h S_0 + O(1)$, we obtain pointwise $f_{T^Y_h}(t) \to P_0(t)/S_0$ on $(-1, 1)$. One more application of Lebesgue's theorem to $\int \varphi(t) f_{T^Y_h}(t)\, dt$ for a continuous bounded $\varphi$ gives the convergence of the integrals, that is, the weak convergence of $T^Y_h$ to the measure with density $P_0(t)/S_0$.

\medskip
\noindent\emph{The polar caps.} It remains to take into account the neighbourhoods of $t = \pm 1$, where $\|\mathbf{n}_2\| = 0$: on these parts of $\partial K_0$ the normal is parallel to $u$ and the integrand in the coarea formula is singular. These are the ``horizontal caps'' of the body, perpendicular to the direction of stretching. Their $(d-1)$-measure on $\partial K_0(h)$ is exactly equal to their $(d-1)$-measure on $\partial K_0$: on a cap $\mathbf{n}_2 = 0$, so in the formula above $\sqrt{n_1^2 + h^2 \cdot 0} = |n_1| = 1$, and the Jacobian of $M_h$ equals $h \cdot 1/h = 1$. Thereby the caps contribute $O(1)$ to the total measure, while the total measure of $\partial K_0(h)$ grows like $h S_0$. Their relative weight in the normalised distribution $T^Y_h$ tends to zero, and the passage to the limit is not affected.

\begin{statement}\label{prop:boundary-proj}
As $h \to \infty$ the distribution of $T^Y_h$ converges weakly to the measure $\mu_Y$ with density $P_0(|t|)/S_0$. Since the support $[-1, 1]$ is compact, this convergence automatically strengthens to the convergence of every polynomial moment.
\end{statement}

\subsubsection{Comparison of the moments of the one-dimensional projections}

Having at hand the limiting measures
\[
    \mu_X(\mathrm{d}t) = \frac{A_0(|t|)}{V_0}\, dt, \qquad
    \mu_Y(\mathrm{d}t) = \frac{P_0(|t|)}{S_0}\, dt,
\]
let us show that under property \propP{} the first is wider than the second. Since both are symmetric with respect to zero, by the definition of width (see~\ref{def:wider}) it suffices to compare the distributions of $|\xi|$ and $|\eta|$ for $\xi \sim \mu_X$, $\eta \sim \mu_Y$. Their distribution functions on $[0, 1]$ are:
\[
    F_X(t) = \frac{2}{V_0} \int_0^t A_0(s)\, ds, \qquad
    F_Y(t) = \frac{2}{S_0} \int_0^t P_0(s)\, ds.
\]

\begin{lemma}[Majorisation of the projections]\label{lem:proj-majorisation}
Under property \propP{} one has
\[
    F_X(t) \;\leq\; F_Y(t) \qquad \text{for all } t \in [0, 1],
\]
with strict inequality on $(0, 1)$. Equivalently, $|\xi|$ strictly stochastically majorises $|\eta|$ on $(0, 1)$.
\end{lemma}

\begin{proof}
Denote $\bar A(t) = \int_0^t A_0(s)\, ds$, $\bar P(t) = \int_0^t P_0(s)\, ds$; then $V_0 = 2 \bar A(1)$, $S_0 = 2 \bar P(1)$. The desired inequality $F_X \leq F_Y$ is written as $\bar A(t)/\bar A(1) \leq \bar P(t)/\bar P(1)$, or, equivalently,
\[
    \frac{\bar A(t)}{\bar P(t)} \;\leq\; \frac{\bar A(1)}{\bar P(1)} \qquad \text{for all } t \in (0, 1],
\]
with strict inequality on $(0, 1)$. It suffices to show that the function $t \mapsto \bar A(t)/\bar P(t)$ is strictly increasing on $(0, 1)$.

A direct computation gives
\[
    \frac{d}{dt} \frac{\bar A(t)}{\bar P(t)} = \frac{A_0(t) \bar P(t) - P_0(t) \bar A(t)}{\bar P(t)^2}.
\]
The sign of the numerator coincides (after division by the positive $P_0(t) \bar P(t)$) with the sign of the difference $A_0(t)/P_0(t) - \bar A(t)/\bar P(t)$. Therefore it suffices to establish
\begin{equation}\label{eq:pointwise-vs-average}
    \frac{\bar A(t)}{\bar P(t)} \;\leq\; \frac{A_0(t)}{P_0(t)} \qquad \text{for } t \in (0, 1],
\end{equation}
with strict inequality on $(0, 1)$.

Let us write the numerator of the left-hand side as the integral of a product of two functions:
\[
    \bar A(t) = \int_0^t \frac{A_0(s)}{P_0(s)} \cdot P_0(s)\, ds.
\]
On the segment $[0, t]$ the function $s \mapsto A_0(s)/P_0(s)$ is increasing (by property \propP), while $s \mapsto P_0(s)$ is non-increasing (for a centrally symmetric convex body the perimeter of a section is maximal at the centre and is non-increasing in $|s|$). Chebyshev's inequality for oppositely monotone integrands gives
\[
    \int_0^t \frac{A_0(s)}{P_0(s)} \cdot P_0(s)\, ds \;\leq\; \frac{1}{t} \int_0^t \frac{A_0(s)}{P_0(s)}\, ds \cdot \int_0^t P_0(s)\, ds.
\]
Dividing by $\bar P(t) = \int_0^t P_0(s)\, ds$, we obtain
\[
    \frac{\bar A(t)}{\bar P(t)} \;\leq\; \frac{1}{t} \int_0^t \frac{A_0(s)}{P_0(s)}\, ds.
\]
Since $A_0/P_0$ is strictly increasing, its mean over $[0, t]$ is strictly less than its value at the right endpoint:
\[
    \frac{1}{t} \int_0^t \frac{A_0(s)}{P_0(s)}\, ds \;<\; \frac{A_0(t)}{P_0(t)} \qquad \text{for } t \in (0, 1).
\]
Combining the last two inequalities, we obtain~\eqref{eq:pointwise-vs-average} with strict inequality on $(0, 1)$. Consequently, $\bar A(t)/\bar P(t)$ is strictly increasing on $(0, 1)$, which completes the proof.
\end{proof}

From the strict stochastic majorisation of the absolute values there follows the strict ordering of the moments of all three types of differences. Let us verify for $\mu_X$, $\mu_Y$ the hypotheses of Proposition~\ref{stmt:strict-three-way}. By Lemma~\ref{lem:proj-majorisation} the measure $\mu_X$ is strictly wider than $\mu_Y$ in the sense of Definition~\ref{def:strictly-wider}. Both measures are continuous (with densities $A_0(|t|)/V_0$ and $P_0(|t|)/S_0$); let us only remark that for finite $h$ the projection of the boundary point has an atom whenever the direction $u$ is perpendicular to a face of the body (the whole face is projected into a single point) --- such as the polar caps $\pm H_a$, perpendicular to $u = e_1$ --- however the weight of such an atom is $O(1/h)$ and vanishes in the limit (see Proposition~\ref{prop:boundary-proj}), so the limiting measure $\mu_Y$ has no atoms. Therefore Proposition~\ref{stmt:strict-three-way} from Section~\ref{sec:linear-symmetric} gives part~(i) of Theorem~\ref{thm:main}:
\[
    C^n_{XX} \;>\; C^n_{XY} \;>\; C^n_{YY} \qquad \text{for all } n \geq 1.
\]

\subsubsection{Comparison of the full moments of the distances}

It remains to pass from the one-dimensional projections to the full Euclidean distances on $K_0(h)$. The stretching along $u$ makes the $u$-component of the displacement of order $h$, whereas the $u^\perp$-component remains bounded independently of $h$. After the normalisation by $h^n$ only the $u$-component survives in the limit.

More precisely: for any $a, b \in K_0(h)$ the orthogonal component $a - b - \pi_u(a - b)\, u$ lies in a set of diameter $R := \mathrm{diam}(K_0)$ independently of $h$. Therefore
\begin{equation}\label{eq:sandwich}
    |\pi_u(a - b)| \;\leq\; |a - b| \;\leq\; \sqrt{\pi_u(a - b)^2 + R^2} \;\leq\; |\pi_u(a - b)| + R,
\end{equation}
where the last inequality follows from $\sqrt{x^2 + y^2} \leq |x| + |y|$ for $x, y \geq 0$. Thus $|a-b|$ is squeezed between $|\pi_u(a-b)|$ and $|\pi_u(a-b)| + R$, and the passage to the limit of the moments reduces to the following technical lemma.

\begin{lemma}[Moment transfer lemma]\label{lem:squeeze-moment}
Let $\{Z_h\}_{h > 0}$ be a family of non-negative random variables and let
\[
    \lim_{h \to \infty} \frac{\E[Z_h^m]}{h^m} = C_m \qquad \text{for each } m > 0.
\]
Then for any $c \geq 0$
\[
    \lim_{h \to \infty} \frac{\E[(Z_h + c)^m]}{h^m} = C_m \qquad \text{for each } m > 0.
\]
\end{lemma}

\begin{proof}
Induction on $k \in \N$ with the statement: the conclusion holds for all $m \leq k$.

\emph{Base $k = 1$.} For $a, c \geq 0$ and $0 < m \leq 1$:
\[
    a^m \;\leq\; (a + c)^m \;\leq\; a^m + c + (1 + c)^m.
\]
The left inequality is trivial. The right one is by cases: for $a \leq 1$ we have $(a + c)^m \leq (1 + c)^m$, and for $a \geq 1$, by the mean value theorem, $(a + c)^m \leq a^m + m c \cdot a^{m-1} \leq a^m + c$. Applying this pointwise with $a = Z_h$, taking the expectation and dividing by $h^m$, we obtain
\[
    \frac{\E[Z_h^m]}{h^m} \;\leq\; \frac{\E[(Z_h + c)^m]}{h^m} \;\leq\; \frac{\E[Z_h^m] + c + (1 + c)^m}{h^m},
\]
and by the squeeze theorem the right limit equals $C_m$.

\emph{Step $k \to k + 1$.} For $a, c \geq 0$ and $m \geq 1$, by the mean value theorem,
\[
    a^m \;\leq\; (a + c)^m \;\leq\; a^m + m c (a + c)^{m-1}.
\]
Therefore for $1 < m \leq k + 1$
\[
    \frac{\E[Z_h^m]}{h^m} \;\leq\; \frac{\E[(Z_h + c)^m]}{h^m} \;\leq\; \frac{\E[Z_h^m] + m c \cdot \E[(Z_h + c)^{m-1}]}{h^m}.
\]
Since $m - 1 \leq k$, by the induction hypothesis $\E[(Z_h + c)^{m-1}] \sim C_{m-1} h^{m-1} = o(h^m)$, and the same passage to the limit gives the required equality.
\end{proof}

Let us apply Lemma~\ref{lem:squeeze-moment} to each of the three pairs.

\medskip
\noindent\emph{The pair of interior points.} Put $Z_h := |\pi_u(X_h - X_h')|$. By Proposition~\ref{prop:interior-proj} the distribution of $T^X_h$ coincides with $\mu_X$ for every $h$, so
\[
    \frac{\E[Z_h^m]}{h^m} \;=\; \E|T^X_h - T^{X'}_h|^m \;=\; C^m_{XX} \qquad \text{for every } h.
\]
The pointwise estimate~\eqref{eq:sandwich} gives $Z_h \leq |X_h - X_h'| \leq Z_h + R$, and the application of Lemma~\ref{lem:squeeze-moment} with $c = R$ gives
\[
    \frac{\E|X_h - X_h'|^n}{h^n} \;\xrightarrow[h \to \infty]{}\; C^n_{XX}.
\]

\medskip
\noindent\emph{The pair of boundary points.} Put $Z_h := |\pi_u(Y_h - Y_h')|$. By Proposition~\ref{prop:boundary-proj} the pair $(T^Y_h, T^{Y'}_h)$ converges weakly to $\mu_Y \otimes \mu_Y$ (by independence and compactness of the supports), and the convergence strengthens to the convergence of the polynomial moments:
\[
    \frac{\E[Z_h^m]}{h^m} \;=\; \E|T^Y_h - T^{Y'}_h|^m \;\xrightarrow[h \to \infty]{}\; C^m_{YY}.
\]
Applying Lemma~\ref{lem:squeeze-moment} with $c = R$, we obtain
\[
    \frac{\E|Y_h - Y_h'|^n}{h^n} \;\xrightarrow[h \to \infty]{}\; C^n_{YY}.
\]

\medskip
\noindent\emph{The mixed pair.} The case $(X_h, Y_h)$ is treated analogously, via the joint convergence of the pair $(T^X_h, T^Y_h)$ to $\mu_X \otimes \mu_Y$:
\[
    \frac{\E|X_h - Y_h|^n}{h^n} \;\xrightarrow[h \to \infty]{}\; C^n_{XY}.
\]

\medskip

This completes part~(ii) of Theorem~\ref{thm:main} and, together with item~(i), its proof.

\subsection{An example of a body with property \propP{} in \texorpdfstring{$\R^3$}{R3}}
\label{sec:body-r3}

In this subsection an explicit centrally symmetric convex body $K_0 \subset \R^3$ with property~\propP{} and distinguished direction $u = e_1$ is constructed. Together with Theorem~\ref{thm:main} and Corollary~\ref{cor:counterexamples} this yields a counterexample to Conjecture~\ref{hyp:zt} in $\R^3$.

\paragraph*{The body.} Fix $a \in (0, 1)$. In $\R^2$ consider
\[
    Q   \;=\; [-1, 1]^2, \qquad
    H_a \;=\; \mathrm{conv}\bigl((\pm 1,\, \pm(1-a)),\; (\pm(1-a),\, \pm 1)\bigr).
\]
The octagon $H_a$ is obtained from $Q$ by cutting off each of the four corners by an isosceles right triangle with leg~$a$, and $H_a \subset Q$. Place $Q$ in the hyperplane $\{x_1 = 0\} \subset \R^3$, and $H_a$ and $-H_a$ in $\{x_1 = +1\}$ and $\{x_1 = -1\}$ respectively, and put
\[
    K_0 \;:=\; \mathrm{conv}\bigl(\, Q \,\cup\, H_a \,\cup\, (-H_a) \,\bigr).
\]
Since $Q$ and $H_a$ are centrally symmetric and $-H_a = H_a$, the body $K_0$ is symmetric with respect to the origin.

\begin{figure}[H]
    \centering
    \includegraphics[width=0.55\textwidth]{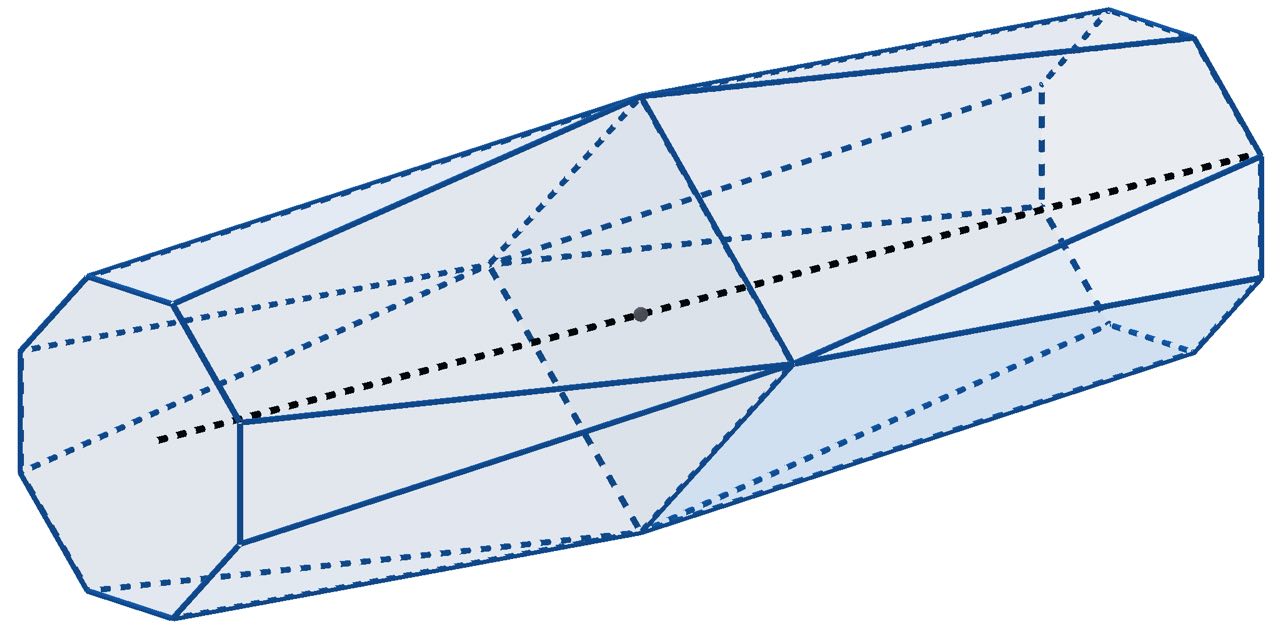}
    \caption{The body $K_0 \subset \R^3$: the convex hull of the square $Q$ in the equatorial section $\{x_1 = 0\}$ and of the two octagons $\pm H_a$ in the polar sections $\{x_1 = \pm 1\}$.}
    \label{fig:body-R3}
\end{figure}

\paragraph*{The sections.} Take $u = e_1$ and denote by $A_0(t)$, $P_0(t)$ the area and the perimeter of the section $K_0(t) = K_0 \cap \{x_1 = t\}$. By central symmetry it suffices to describe $K_0(t)$ for $t \in [0, 1]$.

\begin{statement}[Area and perimeter of a section]\label{prop:body-formulas}
For each $t \in [0, 1]$ the section $K_0(t)$ coincides with the octagon $H_{at}$. In particular,
\[
    A_0(t) \;=\; 4 - 2 a^2 t^2, \qquad
    P_0(t) \;=\; 8 - 4(2 - \sqrt{2})\, a\, t,
\]
and on $[-1, 1]$ the functions are extended by evenness: $A_0(t) = A_0(|t|)$, $P_0(t) = P_0(|t|)$.
\end{statement}

\begin{proof}
At each corner of the square $Q$ the cut of the section $K_0(t)$ depends linearly on the level $t$: at $t = 0$ there is no cut (the square $Q$ itself), and at $t = \pm 1$ the cut equals $a$ (the octagons $\pm H_a$). Therefore at the level $t$ the cut is an isosceles right triangle with leg $at$, and $K_0(t) = H_{at}$.

The octagon $H_s$ has four sides lying on the sides of the square, of length $2(1 - s)$, and four diagonal sides of length $s\sqrt{2}$, whence
\[
    |H_s| \;=\; |Q| - 4 \cdot \tfrac{1}{2} s^2 \;=\; 4 - 2 s^2, \qquad
    P(H_s) \;=\; 4 \cdot 2(1 - s) + 4 \cdot s\sqrt{2} \;=\; 8 - 4(2 - \sqrt{2})\, s.
\]
Substituting $s = at$, we obtain the stated formulas.
\end{proof}

\paragraph*{Verification of property \propP.}

\begin{statement}\label{prop:body-has-P}
Let $a_*$ be the smaller positive root of the equation
\begin{equation}\label{eq:a-star}
    (2 - \sqrt{2})\, a^2 \,-\, 4 a \,+\, 2(2 - \sqrt{2}) \;=\; 0,
\end{equation}
explicitly $a_* = \dfrac{2 - 2\sqrt{2(\sqrt{2} - 1)}}{2 - \sqrt{2}} \approx 0.307$. For any $a \in (0, a_*)$ the body $K_0$ has property~\propP{} with respect to $u = e_1$.
\end{statement}

\begin{proof}
Let us show that $t \mapsto A_0(t)/P_0(t)$ is strictly increasing on $(0, 1)$. On this interval $P_0(t) > 0$, and the sign of the derivative of the ratio coincides with the sign of
\[
    A_0'(t)\, P_0(t) \,-\, A_0(t)\, P_0'(t).
\]
Substituting $A_0'(t) = -4 a^2 t$ and $P_0'(t) = -4(2 - \sqrt{2})\, a$ together with the explicit formulas for $A_0, P_0$, a direct computation shows that this expression equals $8 a \cdot g(t)$, where
\[
    g(t) \;:=\; (2 - \sqrt{2})\, a^2 t^2 \,-\, 4 a t \,+\, 2(2 - \sqrt{2}).
\]
The parabola $g$ opens upwards with vertex at $t_v = 2 / \bigl(a (2 - \sqrt{2})\bigr) > 1$ for all $a \in (0, 1)$, so $g$ is strictly decreasing on $[0, 1]$, and its minimum there is attained at $t = 1$:
\[
    g(1) \;=\; (2 - \sqrt{2})\, a^2 \,-\, 4 a \,+\, 2(2 - \sqrt{2}).
\]
The equation $g(1) = 0$ is exactly~\eqref{eq:a-star}; its discriminant equals $16 - 8(2 - \sqrt{2})^2 = 32(\sqrt{2} - 1) > 0$, and the smaller positive root is $a_*$. For $a \in (0, a_*)$ we have $g(1) > 0$, and since $g$ is decreasing on $[0, 1]$, we have $g > 0$ on the whole segment. Consequently, $A_0/P_0$ is strictly increasing on $(0, 1)$, which is exactly property~\propP.
\end{proof}

Thereby the existence of a centrally symmetric convex body with property~\propP{} in $\R^3$ is established. Combined with Theorem~\ref{thm:main} and Corollary~\ref{cor:counterexamples}, this gives an explicit counterexample to Conjecture~\ref{hyp:zt} in dimension $d = 3$.

\subsection{The method of raising the dimension}
\label{sec:dim-raise}

In this subsection we transfer the strict ordering of the three normalised moments $\E|X-X'|^n$, $\E|X-Y|^n$, $\E|Y-Y'|^n$ from dimension $m$ to dimension $m+1$ via the Cartesian product of the body with a segment. The boundary of the product splits into \emph{caps} and the \emph{lateral part}, and the three moments of the new family are expressed through the moments of the original one by conditioning on the type of the point.

\paragraph*{The structure of the Cartesian product.} Let $K \subset \R^m$ ($m \geq 3$) be a centrally symmetric convex body, $L > 0$. Put
\[
    \tilde K \;:=\; K \times [-L, L] \;\subset\; \R^{m+1}.
\]
Let $X, X'$ (respectively $Y, Y'$) be independent uniformly distributed interior (respectively boundary) points of $K$, and let $\tilde X, \tilde X'$ (respectively $\tilde Y, \tilde Y'$) be their analogues for $\tilde K$. The boundary of the product splits into two parts:
\[
    \partial \tilde K
    \;=\; \underbrace{K \times \{\pm L\}}_{\text{caps, area } 2|K|}
    \;\sqcup\;
    \underbrace{\partial K \times [-L, L]}_{\text{side, area } 2L \cdot |\partial K|}.
\]
A random boundary point $\tilde Y$ falls on the caps with probability
\[
    q \;:=\; \frac{|K|}{|K| + L \cdot |\partial K|},
\]
and on the lateral part with probability $1 - q$. Conditionally on $\{\text{cap}\}$ we have $\tilde Y = (B^{\mathrm{c}}, \pm L)$, where $B^{\mathrm{c}}$ is uniformly distributed \emph{inside} $K$ (the first coordinate is of interior type), and $\pm L$ are equally likely. Conditionally on $\{\text{side}\}$ we have $\tilde Y = (B^{\mathrm{l}}, V)$, where $B^{\mathrm{l}}$ is uniform \emph{on} $\partial K$ (the first coordinate is of boundary type), and $V$ is uniform on $[-L, L]$. An interior point has the form $\tilde X = (X, U)$, where $X$ is an interior point of $K$ and $U$ is uniform on $[-L, L]$, everything being independent.

We shall need an elementary chain of inequalities: for any $\tilde p_i = (x_i, u_i) \in \tilde K$, from $\sqrt{a^2 + b^2} \leq |a| + |b|$ it follows that
\begin{equation}\label{eq:sandwich-product}
    |x_1 - x_2| \;\leq\; |\tilde p_1 - \tilde p_2| \;\leq\; |x_1 - x_2| + 2L.
\end{equation}

\paragraph*{Transfer of the limiting family.}

\begin{statement}\label{prop:method-A}
Let $\{K_m(h)\}_{h > 0}$ be a family of centrally symmetric convex bodies in $\R^m$ ($m \geq 3$). Suppose that for each $n \geq 1$ there exist constants $C^n_{XX}, C^n_{XY}, C^n_{YY}$ with the strict ordering $C^n_{XX} > C^n_{XY} > C^n_{YY}$ such that
\begin{equation}\label{eq:family-hyp}
    \frac{\E|X-X'|^n}{h^n} \to C^n_{XX}, \quad
    \frac{\E|X-Y|^n}{h^n} \to C^n_{XY}, \quad
    \frac{\E|Y-Y'|^n}{h^n} \to C^n_{YY}
\end{equation}
as $h \to \infty$, and suppose in addition that
\[
    q(h) \;:=\; \frac{|K_m(h)|}{|K_m(h)| + |\partial K_m(h)|}
    \;\xrightarrow[h \to \infty]{}\; q^* \in (0, 1).
\]
Then the family $\tilde K_{m+1}(h) := K_m(h) \times [-1, 1]$ satisfies the same conditions in dimension $m + 1$. Namely: its three normalised moments converge to
\begin{equation}\label{eq:method-A-recursion}
\begin{aligned}
    \tilde C^n_{XX} \;&=\; C^n_{XX}, \\
    \tilde C^n_{XY} \;&=\; q^*\, C^n_{XX} \,+\, (1 - q^*)\, C^n_{XY}, \\
    \tilde C^n_{YY} \;&=\; (q^*)^2\, C^n_{XX} \,+\, 2 q^* (1 - q^*)\, C^n_{XY} \,+\, (1 - q^*)^2\, C^n_{YY},
\end{aligned}
\end{equation}
with the strict ordering $\tilde C^n_{XX} > \tilde C^n_{XY} > \tilde C^n_{YY}$, and the cap probability $\tilde q(h)$ for $\tilde K_{m+1}(h)$ has a limit in $(0, 1)$.
\end{statement}

\begin{proof}
Take $L = 1$ and denote $K = K_m(h)$, $\tilde K = K \times [-1, 1]$. The last coordinate contributes to~\eqref{eq:sandwich-product} a term bounded by a constant, so applying Lemma~\ref{lem:squeeze-moment} (the moment transfer lemma) with $c = 2$ to $Z_h = |x_1 - x_2|$ gives: if $\E|x_1 - x_2|^n / h^n$ converges to some limit, then $\E|\tilde p_1 - \tilde p_2|^n / h^n$ converges to the same limit.

\emph{The interior pair.} The first coordinates of $(\tilde X, \tilde X')$ coincide with $(X, X')$, so $\E|\tilde X - \tilde X'|^n / h^n \to C^n_{XX}$.

\emph{The mixed pair.} Conditioning on the type of $\tilde Y$ gives
\[
    \frac{\E|\tilde X - \tilde Y|^n}{h^n}
    \;=\; q(h) \cdot \frac{\E|\tilde X - (B^{\mathrm{c}}, \pm 1)|^n}{h^n}
      \;+\; (1 - q(h)) \cdot \frac{\E|\tilde X - (B^{\mathrm{l}}, V)|^n}{h^n}.
\]
The first coordinate of $(B^{\mathrm{c}}, \pm 1)$ is uniform \emph{inside} $K$, so the first summand converges to $C^n_{XX}$; for the lateral condition it converges to $C^n_{XY}$. Taking into account $q(h) \to q^*$:
\[
    \frac{\E|\tilde X - \tilde Y|^n}{h^n}
    \;\xrightarrow[h \to \infty]{}\;
    q^* C^n_{XX} + (1 - q^*) C^n_{XY}.
\]

\emph{The pair of boundary points.} Conditioning on the types of both $\tilde Y$ and $\tilde Y'$ gives three configurations (cap--cap, cap--side, side--side) with weights $q(h)^2$, $2 q(h)(1 - q(h))$, $(1 - q(h))^2$; the marginals of the first coordinate for them are interior--interior, interior--boundary, boundary--boundary. An analogous passage gives
\[
    \frac{\E|\tilde Y - \tilde Y'|^n}{h^n}
    \;\xrightarrow[h \to \infty]{}\;
    (q^*)^2 C^n_{XX} + 2 q^* (1 - q^*) C^n_{XY} + (1 - q^*)^2 C^n_{YY}.
\]

\emph{The cap probability.} The volume and the surface area of $\tilde K_{m+1}(h) = K_m(h) \times [-1, 1]$ are expressed through $|K_m(h)|$ and $|\partial K_m(h)|$:
\[
    |\tilde K_{m+1}(h)| \;=\; 2 |K_m(h)|, \qquad
    |\partial \tilde K_{m+1}(h)| \;=\; 2 |K_m(h)| + 2 |\partial K_m(h)|.
\]
The condition $q(h) \to q^* \in (0, 1)$ is equivalent to the ratio $|K_m(h)| / |\partial K_m(h)|$ having a finite positive limit $q^* / (1 - q^*)$. Substituting into the expression for the cap probability,
\[
    \tilde q(h) \;=\; \frac{|\tilde K_{m+1}(h)|}{|\tilde K_{m+1}(h)| + |\partial \tilde K_{m+1}(h)|}
    \;=\; \frac{|K_m(h)|}{2 |K_m(h)| + |\partial K_m(h)|}
    \;\xrightarrow[h \to \infty]{}\; \frac{q^*}{1 + q^*} \;\in\; (0, 1).
\]

\emph{The strict ordering.} From~\eqref{eq:method-A-recursion}, by a direct computation,
\[
    \tilde C^n_{XX} - \tilde C^n_{XY}
    = (1 - q^*)(C^n_{XX} - C^n_{XY}) > 0,
\]
\[
    \tilde C^n_{XY} - \tilde C^n_{YY}
    = q^* (1 - q^*)(C^n_{XX} - C^n_{XY}) + (1 - q^*)^2 (C^n_{XY} - C^n_{YY}) > 0,
\]
by the input ordering and $q^* \in (0, 1)$.
\end{proof}

Two observations follow from Proposition~\ref{prop:method-A}.

\smallskip
\emph{(i) The method iterates.} The output family $\{\tilde K_{m+1}(h)\}$ satisfies in dimension $m+1$ all three input conditions of Proposition~\ref{prop:method-A}: the convergence of the three normalised moments, the strict ordering of the limits and the convergence of the cap probability in $(0, 1)$. Therefore the proposition is applicable to it again --- and so inductively for every dimension $m + k$, $k \geq 1$.

\smallskip
\emph{(ii) The family in $\R^3$ satisfies the conditions.} The family $\{K_0(h)\}_{h > 0}$ from Subsection~\ref{sec:body-r3} satisfies all three conditions of Proposition~\ref{prop:method-A} for $m = 3$: the convergence~\eqref{eq:family-hyp} is part~(ii) of Theorem~\ref{thm:main}, the strict ordering is its part~(i), and the cap probability $q(h) = |K_0(h)| / (|K_0(h)| + |\partial K_0(h)|)$ converges to a limit in $(0, 1)$, since both $|K_0(h)|$ and $|\partial K_0(h)|$ grow linearly in $h$ (an anisotropic stretching of a fixed base body).

\paragraph*{Counterexamples in every dimension.}

\begin{corollary}\label{cor:counterexamples-alld}
For any $d \geq 3$ and $n \geq 1$ there is a centrally symmetric convex body $K \subset \R^d$ such that
\[
    \E|X - X'|^n \;>\; \E|X - Y|^n \;>\; \E|Y - Y'|^n.
\]
In particular, Conjecture~\ref{hyp:zt} fails in every such dimension.
\end{corollary}

\begin{proof}
The base case $d = 3$ is provided by Proposition~\ref{prop:body-has-P} combined with Theorem~\ref{thm:main} and Corollary~\ref{cor:counterexamples}. For $d \geq 4$ the family $\{K_0(h)\}$ from Subsection~\ref{sec:body-r3} satisfies the conditions of Proposition~\ref{prop:method-A} (observation (ii) above). Applying it $d - 3$ times, we obtain a family $\{K(h)\}$ in $\R^d$ with a strict ordering of the limits of the normalised moments. Applying Corollary~\ref{cor:counterexamples} to this family for $h$ large enough gives the required body.
\end{proof}

\section{The general case of higher-order moments}
This section will be devoted to the proof of the results obtained for planar bodies without additional restrictions.
\subsection{Preparation}

In this subsection we shall make active use of the formulas obtained by T.~Moseeva~\cite{Moseeva2022} on the basis of the Crofton formulas and the Blaschke--Petkantschin formula. The most general form of the formula:

\begin{theorem}[T.~Moseeva, 2022]
Let $K$ be a convex body in $\R^d$ with smooth boundary, let $h(x_0, \ldots, x_l)$ be a continuous function, $l \leq d$. Then for all $k \in \{0, \ldots, l+1\}$ we have
\begin{align*}
    &\int_{(\partial K)^k} \int_{K^{l-k+1}}
    h(x_0, \ldots, x_l)\, dx_0 \ldots dx_{l-k}\, \sigma(dx_{l-k+1}) \ldots \sigma(dx_l) \\
    &= (l!)^{d-l} b_{d,l}
    \int_{A_{d,l}} \int_{(E \cap \partial K)^k} \int_{(E \cap K)^{l-k+1}}
    h(x_0, \ldots, x_l)\, [x_0, \ldots, x_l]^{d-l} \\
    &\qquad \times \prod_{j=l-k+1}^l \| P_E(n_K(x_j)) \|^{-1}
    \lambda_E(dx_0) \ldots \lambda_E(dx_{l-k})\,
    \sigma_{E \cap \partial K}(dx_{l-k+1}) \ldots \sigma_{E \cap \partial K}(dx_l)\,
    \mu_{d,l}(dE).
\end{align*}
Here $b_{d,l} = \frac{\omega_{d-l+1} \cdots \omega_d}{\omega_1 \cdots \omega_l}$, $\omega_k := |\S^{k-1}|$, $n_K(x)$ denotes the outer normal to $K$ at the point $x$, and $P_E$ is the orthogonal projection onto the plane $E$.
\end{theorem}

In our case $d = 2$, $l = 1$, and $k \in \{0, 1, 2\}$ will vary depending on the formula. The function $h$ we shall everywhere take to be the same:
$$h_n(x_0, x_1) = |x_0 x_1|^n \cdot \Ind\{x_0, x_1 \in K\}.$$

Let us use this formula with $k = 0$:
\begin{align*}
    \E_K |XX'|^n
    &= \cfrac{1}{|K|^2} \int_{K^2} h_n(x_0, x_1)\, dx_0\, dx_1 \\
    &= \cfrac{1}{|K|^2} \cdot b_{2,1} \int_{A_{2,1}} \int_{(E \cap K)^2} |x_0 x_1|^{n+1}\, \lambda_E(dx_0)\, \lambda_E(dx_1)\, \mu_{2,1}(dE) \\
    &= \cfrac{1}{|K|^2} \cdot b_{2,1} \int_{A_{2,1}} |E \cap K|^2 \cdot \E_{E \cap K} |x_0 x_1|^{n+1}\, \mu_{2,1}(dE) \\
    &= \cfrac{1}{|K|^2} \cdot b_{2,1} \int_{A_{2,1}} |E \cap K|^2 \cdot \cfrac{2 |E \cap K|^{n+1}}{(n+2)(n+3)}\, \mu_{2,1}(dE) \\
    &= \cfrac{1}{|K|^2} \cdot \cfrac{2 b_{2,1}}{(n+2)(n+3)} \int_{A_{2,1}} |E \cap K|^{n+3}\, \mu_{2,1}(dE) \\
    &= \cfrac{1}{|K|^2} \cdot \cfrac{\omega_2}{(n+2)(n+3)} \int_{A_{2,1}} |E \cap K|^{n+3}\, \mu_{2,1}(dE).
\end{align*}

Another expression that we shall need:
\begin{align*}
    \E_K |XY|^n
    &= \cfrac{1}{|K| \cdot |\sigma K|} \int_K \int_{\sigma K} |x_0 x_1|^n\, dx_0\, \sigma(dx_1) \\
    &= \cfrac{1}{|K| \cdot |\sigma K|} \cdot \cfrac{\omega_2}{4(n+2)} \int_{A_{2,1}} |E \cap K|^{n+2} \left(\cfrac{1}{\sin\alpha_1} + \cfrac{1}{\sin\alpha_2}\right) \mu_{2,1}(dE).
\end{align*}

It remains to express only one quantity:
\begin{align*}
    \E_K |YY'|^n
    &= \cfrac{1}{|\sigma K|^2} \int_{(\sigma K)^2} |x_0 x_1|^n\, \sigma(dx_0)\, \sigma(dx_1) \\
    &= \cfrac{1}{|\sigma K|^2} \cdot \cfrac{\omega_2}{2} \int_{A_{2,1}} |E \cap K|^{n+1} \cdot \cfrac{1}{\sin\alpha_1 \cdot \sin\alpha_2}\, \mu_{2,1}(dE).
\end{align*}

In what follows, in order to prove the results we shall perform simple manipulations with these formulas. The key, but nevertheless simple, consideration will be the following inequality:
$$|E \cap K| \leq \cfrac{|\sigma K|}{2}.$$

\subsection{Sufficient conditions for $\E_K|XX'|^n \leq \E_K|YY'|^n$}

\begin{align*}
    \E_K |XX'|^n
    &= \cfrac{1}{|K|^2} \cdot \cfrac{\omega_2}{(n+2)(n+3)} \int_{A_{2,1}} |E \cap K|^{n+3}\, \mu_{2,1}(dE) \\
    &\leq \cfrac{1}{|K|^2} \cdot \cfrac{\omega_2}{(n+2)(n+3)} \cdot \cfrac{|\sigma K|^2}{4} \int_{A_{2,1}} |E \cap K|^{n+1}\, \mu_{2,1}(dE) \\
    &\leq \cfrac{1}{|K|^2} \cdot \cfrac{\omega_2}{(n+2)(n+3)} \cdot \cfrac{|\sigma K|^2}{4} \int_{A_{2,1}} \cfrac{|E \cap K|^{n+1}}{\sin\alpha_1 \cdot \sin\alpha_2}\, \mu_{2,1}(dE) \\
    &= \cfrac{1}{|K|^2} \cdot \cfrac{\omega_2}{(n+2)(n+3)} \cdot \cfrac{|\sigma K|^2}{4} \cdot \E_K |YY'|^n \cdot |\sigma K|^2 \cdot \cfrac{2}{\omega_2} \\
    &= \E_K |YY'|^n \cdot \cfrac{|\sigma K|^4}{|K|^2} \cdot \cfrac{1}{2(n+2)(n+3)}.
\end{align*}

For the inequality $\E_K|XX'|^n \leq \E_K|YY'|^n$ it suffices that
$$\cfrac{|K|}{|\sigma K|^2} \geq \cfrac{1}{\sqrt{2(n+2)(n+3)}}.$$

\subsection{Sufficient conditions for $\E_K|XX'|^n \leq \E_K|XY|^n$}

\begin{align*}
    \E_K |XX'|^n
    &= \cfrac{1}{|K|^2} \cdot \cfrac{\omega_2}{(n+2)(n+3)} \int_{A_{2,1}} |E \cap K|^{n+3}\, \mu_{2,1}(dE) \\
    &\leq \cfrac{1}{|K|^2} \cdot \cfrac{\omega_2}{(n+2)(n+3)} \cdot \cfrac{|\sigma K|}{2} \int_{A_{2,1}} |E \cap K|^{n+2}\, \mu_{2,1}(dE) \\
    &\leq \cfrac{1}{|K|^2} \cdot \cfrac{\omega_2}{(n+2)(n+3)} \cdot \cfrac{|\sigma K|}{4} \int_{A_{2,1}} |E \cap K|^{n+2} \left(\cfrac{1}{\sin\alpha_1} + \cfrac{1}{\sin\alpha_2}\right) \mu_{2,1}(dE) \\
    &= \cfrac{1}{|K|^2} \cdot \cfrac{\omega_2}{(n+2)(n+3)} \cdot \cfrac{|\sigma K|}{4} \cdot \E_K |XY|^n \cdot |\sigma K| \cdot |K| \cdot \cfrac{4(n+2)}{\omega_2} \\
    &= \E_K |XY|^n \cdot \cfrac{|\sigma K|^2}{|K|} \cdot \cfrac{1}{n+3}.
\end{align*}

For the inequality $\E_K|XX'|^n \leq \E_K|XY|^n$ it suffices that
$$\cfrac{|K|}{|\sigma K|^2} \geq \cfrac{1}{n+3}.$$

\subsection{Sufficient conditions for $\E_K|XY|^n \leq \E_K|YY'|^n$}

\begin{align*}
    \E_K |XY|^n
    &= \cfrac{1}{|K| \cdot |\sigma K|} \cdot \cfrac{\omega_2}{4(n+2)} \int_{A_{2,1}} |E \cap K|^{n+2} \left(\cfrac{1}{\sin\alpha_1} + \cfrac{1}{\sin\alpha_2}\right) \mu_{2,1}(dE) \\
    &\leq \cfrac{1}{|K| \cdot |\sigma K|} \cdot \cfrac{\omega_2}{4(n+2)} \cdot \cfrac{|\sigma K|}{2} \int_{A_{2,1}} |E \cap K|^{n+1} \left(\cfrac{1}{\sin\alpha_1} + \cfrac{1}{\sin\alpha_2}\right) \mu_{2,1}(dE) \\
    &\leq \cfrac{1}{|K| \cdot |\sigma K|} \cdot \cfrac{\omega_2}{4(n+2)} \cdot |\sigma K| \int_{A_{2,1}} \cfrac{|E \cap K|^{n+1}}{\sin\alpha_1 \cdot \sin\alpha_2}\, \mu_{2,1}(dE) \\
    &= \cfrac{1}{|K| \cdot |\sigma K|} \cdot \cfrac{\omega_2}{4(n+2)} \cdot |\sigma K| \cdot |\sigma K|^2 \cdot \cfrac{2}{\omega_2} \cdot \E_K |YY'|^n \\
    &= \E_K |YY'|^n \cdot \cfrac{|\sigma K|^2}{|K|} \cdot \cfrac{1}{2(n+2)}.
\end{align*}

For the inequality $\E_K|XY|^n \leq \E_K|YY'|^n$ it suffices that
$$\cfrac{|K|}{|\sigma K|^2} \geq \cfrac{1}{2(n+2)}.$$

One may observe that all the constants tend to zero as the index of the moment grows, which means: for \textit{any} convex body, starting from some moment, the quantities under study are ordered in the stated way. It is also worth emphasising that the conditions for small $n$ are rather useless, since they require the violation of the isoperimetric inequality.

\section{Projective distributions and exact relations between the quantities}

We shall devote this section to the search for more precise relations. We shall be interested in the conditions under which one can compute exactly the ratio between the moments of the mean distances. Unfortunately, it turned out that results of this kind require a completely different technique and much more rigid restrictions. However, along the way interesting results arise for objects that have not been studied before.

The main difficulty arises in the attempt to find a parameter along which it would be convenient to \textit{``decompose''} our construction. Most of the standard quantities, such as the Cartesian coordinates or the angles with various directions, have densities and mutual dependencies that are hard to express, which complicates the analysis.

The first goal will be the construction of a parametrisation of the points inside a convex body that is convenient for further computations.

\subsection{Homothetic coordinates}
\label{sec:homothetic}

In what follows we shall use the Minkowski functional, so we give its definition here.

\begin{definition}[Minkowski functional]
For any vector space $X$ and its subset $K$ the Minkowski functional $\mu_K \colon X \to [0, \infty)$ is defined as
$$\mu_K(x) = \inf \left\{r > 0 : x \in rK\right\}.$$
\end{definition}

Now we are ready to introduce our parametrisation.
\begin{definition}
Let some convex figure $K \subset \R^d$ and a point $O \in K$ be fixed. Then any point $T \in K$, $T \neq O$, is specified by a pair $(\alpha, s)$, where $\alpha \in [0, 1]$, $s \in S^{d-1}$, as follows. The coordinate $s$ determines the direction of the vector $OT$; let the ray $OT$ meet $\partial K$ at the point $T_Y$, then $\alpha := \cfrac{|OT|}{|OT_Y|}$ is the Minkowski functional of the point $T$ in the coordinate system with origin at $O$ relative to the set $K$. In essence, $\alpha$ is the coefficient of the homothety with centre at the point $O$ under which the point $T$ falls on the image of the boundary.
\end{definition}

The parametrisation obtained has a rectangular domain of definition, which is already good, but in fact it has more good properties.
\begin{statement}
Let, as before, the point $X$ be distributed inside $K$ proportionally to its Lebesgue measure. Then its coordinates $X_\alpha$ and $X_s$ are independent. Moreover, the distribution of $X_\alpha$ has density $p(t) = d \cdot t^{d-1} \cdot \Ind_{[0,1]}(t)$.
\end{statement}
\begin{proof}
First let us show that the distribution of the $\alpha$-coordinate coincides with the stated one. Note that
\[
    \P(X_\alpha \leq t) = t^d.
\]
Indeed, the inequality $X_\alpha \leq t$ means that the point $X$ has fallen into the image of $K$ under the homothety with centre at $O$ and coefficient $t$; this image has volume $t^d$ times smaller than the original figure. From this we obtain that $X_\alpha$ has the stated density.

It remains to establish the independence. For this it suffices to show that
\begin{equation}\label{Lotnikov:Xndependancy1}
    \P(X_s \in A \mid X_\alpha \leq t) = \P(X_s \in A) \qquad \text{for all } t \in [0, 1],\; A \subset S^{d-1}.
\end{equation}

Indeed, by the definition of conditional probability
\[
    \P(X_s \in A \mid X_\alpha \leq t) = \cfrac{\P(X_s \in A,\; X_\alpha \leq t)}{\P(X_\alpha \leq t)}.
\]
Denote by $K_A$ the cone with apex at the point $O$ and base $A$, and by $H_{O, t}(K)$ the image of $K$ under the homothety with centre at the point $O$ and coefficient $t$. Then
\[
    \cfrac{\P(X_s \in A,\; X_\alpha \leq t)}{\P(X_\alpha \leq t)} = \cfrac{\nu_d(K_A \cap H_{O, t}(K))}{\nu_d(H_{O, t}(K))} = \cfrac{\nu_d(K_A \cap K)}{\nu_d(K)} = \P(X_s \in A).
\]

The penultimate passage is valid because under the homothety with centre at the point $O$ and coefficient $1/t$ the set $K_A \cap H_{O, t}(K)$ goes over into $K_A \cap K$, and $H_{O, t}(K)$ into $K$, the volumes changing by the same factor, so that the ratio is preserved.
\end{proof}

\noindent Now, knowing such a decomposition, we can prove the announced theorem on projective distributions.

\subsection{Proof}

First let us verify the last item of the theorem. Recall that the lower index $\alpha$ means taking the corresponding coordinate of a point inside $K$.

\[
\E|XX'|^n = \E\left[\E[|XX'|^n \mid (X_\alpha, X'_\alpha)]\right]
\]

\[
= \int_0^1\int_0^1 \E[|XX'|^n \mid X_\alpha = t_1, X'_\alpha = t_2] \cdot d^2 \cdot t_1^{d-1}t_2^{d-1} \, dt_1dt_2
\]
From the properties of homothety and the independence of the $\alpha$ and $t$ coordinates

\[ \forall t_1 < t_2 \quad \E[|XX'|^n \mid X_\alpha = t_1, X'_\alpha = t_2] = t_2^n \cdot \E\!\left[|XX'|^n \mid X_\alpha = \cfrac{t_1}{t_2},\; X'_\alpha = 1 \right].
\]

which allows us to continue the chain of equalities:
\begin{align*}
    \E|XX'|^n
    &= 2 \cdot \int_0^1\int_0^{t_2} \E[|XX'|^n \mid X_\alpha = t_1, X'_\alpha = t_2] \cdot d^2 \cdot t_1^{d-1}t_2^{d-1} \, dt_1dt_2 \\
    &= 2 \cdot \int_0^1\int_0^{t_2} \E[|XX'|^n \mid X_\alpha = \cfrac{t_1}{t_2}, X'_\alpha = 1] \cdot t_2^n \cdot d^2 \cdot t_1^{d-1}t_2^{d-1} \, dt_1dt_2 \\
    &= 2 \cdot \int_0^1\int_0^{t_2} \E[|XX'|^n \mid X_\alpha = \cfrac{t_1}{t_2}, X'_\alpha = 1] \cdot d^2 \cdot \left(\cfrac{t_1}{t_2}\right)^{d-1} \cdot t_2^{n+2d-1} \, \cfrac{dt_1}{t_2} \cdot dt_2 \\
    &= 2d \cdot \int_0^1\int_0^1 \E[|XX'|^n \mid X_\alpha = t_1, X'_\alpha = 1] \cdot d \cdot t_1^{d-1} \cdot t_2^{n+2d-1} \, dt_1\, dt_2 \\
    &= \cfrac{2d}{n+2d} \cdot \int_0^1 \E[|XX'|^n \mid X_\alpha = t_1, X'_\alpha = 1] \cdot d \cdot t_1^{d-1} \, dt_1 \\
    &= \cfrac{2d}{n+2d} \cdot \E|XY_O|^n.
\end{align*}

% $$ $$
% $$ 2 \cdot \int_0^1\int_0^{t_2} \E[|XX'|^p \mid X_\alpha = \cfrac{t_1}{t_2}, X'_\alpha = 1] \cdot 2 \cdot \cfrac{t_1}{t_2} \cdot t_2^4 \, \cfrac{dt_1}{t_2}dt_2 = $$
% $$ 2 \cdot \int_0^1\int_0^1 \E[|XX'|^p \mid X_\alpha = t_1, X'_\alpha = 1] \cdot 2 \cdot t_1t_2^{p+3} \, dt_1dt_2 = $$
% $$ 2 \cdot \int_0^1\int_0^1 \E[|XX'|^p \mid X_\alpha = t_1, X'_\alpha = 1] \cdot 2 \cdot t_1t_2^{p+3} \, dt_2dt_1 = $$
% $$ \cfrac{2}{p+4} \cdot \int_0^1\int_0^1 \E[|XX'|^p \mid X_\alpha = t_1, X'_\alpha = 1] \cdot 2 t_1 \, dt_1 = $$
% $$ \cfrac{2}{p+4} \cdot \E|XY_0| $$

% \noindent Statement 2.1 follows from what has been proved. Our quantities are bounded a priori, and hence are determined by their moments. At the same time, by the first item the moments of the quantity $|XY_O|$ are expressed through the moments of $|XX'|$, and hence do not depend on the choice of the point $O$. Consequently, neither do the distributions. \newline

\noindent The second item is also practically a direct consequence of the first. Indeed, as we know, for any bounded random variable $X$ its characteristic function $\varphi$ is such that:
\[
    \varphi(t) = \E e^{itX} = \sum\limits_{k=0}^{\infty} \cfrac{\E X^k \cdot i^k}{k!} \cdot t^k.
\]
In particular,
\begin{align*}
    \mu_{XY_O}(t) := \E e^{it|XY_O|}
    &= \sum\limits_{k=0}^{\infty} \cfrac{\E|XY_O|^k \cdot i^k}{k!} \cdot t^k
     = \sum\limits_{k=0}^{\infty} \cfrac{\E|XX'|^k \cdot i^k}{k!} \cdot \cfrac{k+2d}{2d} \cdot t^k \\
    &= \cfrac{1}{t^{2d-1} \cdot 2d} \sum\limits_{k=0}^{\infty} \cfrac{\E|XX'|^k \cdot i^k}{k!} \cdot (k+2d) \cdot t^{k+2d-1} \\
    &= \cfrac{1}{t^{2d-1} \cdot 2d} \cdot \cfrac{d}{dt}\left[\sum\limits_{k=0}^{\infty} \cfrac{\E|XX'|^k \cdot i^k}{k!} \cdot t^{k+2d}\right] \\
    &= \cfrac{1}{t^{2d-1} \cdot 2d} \cdot \cfrac{d}{dt}\left[\E e^{it|XX'|} \cdot t^{2d}\right].
\end{align*}

\noindent The verification of the third item is carried out analogously to the verification of the first. By the properties of homothety $\E[|OX|^n \mid X_\alpha = t] = \E[|OX|^n \mid X_\alpha = 1] \cdot t^n$, therefore
\begin{align*}
    \E|OX|^n
    &= \E\left[\E[|OX|^n \mid X_\alpha]\right]
     = \int_0^1 \E[|OX|^n \mid X_\alpha = t] \cdot d \cdot t^{d-1} \, dt \\
    &= \int_0^1 \E[|OX|^n \mid X_\alpha = 1] \cdot d \cdot t^{d+n-1} \, dt \\
    &= \cfrac{d}{n+d} \cdot \E[|OX|^n \mid X_\alpha = 1]
     = \cfrac{d}{n+d} \cdot \E|OY_O|^n.
\end{align*}
The relation between the characteristic functions is derived analogously.

From the results described one can also obtain a rather unexpected corollary.
\begin{corollary}
The distribution of $|XY_O|$ does not depend on the choice of the point $O$.
\end{corollary}
\begin{proof}
    We know that
    $$\E|XY_O|^n  = \cfrac{n+2d}{2d} \cdot \E|XX'|^n .$$
    That is, the moments of the distribution of $|XY_O|$ do not depend on the choice of the point $O$, but our distributions are bounded a priori. Consequently, the distributions coincide as well.
\end{proof}

\subsection{On the distributions $Y_O$}
Note that the distributions $Y_O$ in the case of convex polytopes are arranged in a rather simple way:
\begin{enumerate}
    \item They are concentrated on the faces of dimension $d-1$.
    \item The distribution inside each face is uniform.
    \item The uniform distributions inside the faces are mixed with weights proportional to the volume of the $d$-dimensional pyramid with apex at the point $O$ and base on the corresponding face.
\end{enumerate}

\noindent  In particular, this means that $Y_O \overset{d}{=} Y$ if and only if there exists a sphere inscribed in the polytope and $O$ is its centre. \newline

That is, provided that our figure is circumscribed, we can replace all the projective distributions by the ordinary uniform distributions on the boundary.

\subsection{Relation with the mean distance between two points on the boundary}

We have obtained an exact expression for the moments of the distance between two interior points in terms of the moments of the distance between an interior and a boundary point, which is in itself very interesting, but the goal was originally the relation with the distance between two boundary points.
This relation is stated in the following theorem.

\begin{theorem}
    Let $K \subset \R^d$ be a convex body, and let $O$ be some interior point of it. Then
    \begin{align*}
        \E|XX'|^2 &= \cfrac{2d}{d+2} \cdot \E|OY_O|^2 - 2 \cdot \cfrac{d^2}{(d+1)^2} \cdot |O\E Y_O|^2, \\
        \E|XY_O|^2 &= \left(1 + \cfrac{d}{d+2}\right) \cdot \E|OY_O|^2 - 2 \cdot \cfrac{d}{d+1} \cdot |O\E Y_O|^2, \\
        \E|Y_O Y_O'|^2 &= 2 \cdot \E|OY_O|^2 - 2 \cdot |O\E Y_O|^2.
    \end{align*}
\end{theorem}
\begin{proof}
    Let
    $$f(\alpha, \beta) = \E\bigl[|XX'|^2 \mid X_\alpha = \alpha,\, X'_\alpha = \beta\bigr] = \int_{\S^{d-1}}\int_{\S^{d-1}} |\alpha \cdot \gamma(t) - \beta \cdot \gamma(s)|^2 \, d\mu_S(s)\, d\mu_S(t),$$
    where $\gamma(t)$ is the point with homothetic coordinates $(1, t)$, and $\mu_S$ is the distribution of the $s$-coordinate in our body. Let us expand the square of the Euclidean norm:
    \begin{align*}
        f(\alpha, \beta)
        &= \int_{\S^{d-1}}\int_{\S^{d-1}} \langle \alpha \gamma(t) - \beta \gamma(s),\, \alpha \gamma(t) - \beta \gamma(s) \rangle \, d\mu_S(s)\, d\mu_S(t) \\
        &= \int_{\S^{d-1}}\int_{\S^{d-1}} \bigl[\alpha^2 \langle\gamma(t), \gamma(t)\rangle - 2 \alpha\beta \langle\gamma(t), \gamma(s)\rangle + \beta^2 \langle\gamma(s), \gamma(s)\rangle\bigr] \, d\mu_S(s)\, d\mu_S(t) \\
        &= (\alpha^2 + \beta^2) \int_{\S^{d-1}} \langle\gamma(t), \gamma(t)\rangle \, d\mu_S(t) - 2 \alpha\beta \int_{\S^{d-1}}\int_{\S^{d-1}} \langle\gamma(t), \gamma(s)\rangle \, d\mu_S(s)\, d\mu_S(t) \\
        &= (\alpha^2 + \beta^2) \cdot \E|OY_O|^2 - 2 \alpha\beta \cdot |O\E Y_O|^2.
    \end{align*}

    Now all three quantities sought are expressed through $f$.

    Between two interior points --- we integrate over both radii:
    \begin{align*}
        \E|XX'|^2
        &= \int_0^1 \int_0^1 f(\alpha, \beta) \cdot d \cdot \alpha^{d-1} \cdot d \cdot \beta^{d-1}\, d\alpha\, d\beta \\
        &= \int_0^1 \int_0^1 \bigl[(\alpha^2 + \beta^2) \cdot \E|OY_O|^2 - 2 \alpha\beta \cdot |O\E Y_O|^2\bigr] \cdot d \cdot \alpha^{d-1} \cdot d \cdot \beta^{d-1}\, d\alpha\, d\beta \\
        &= \cfrac{2d}{d+2} \cdot \E|OY_O|^2 - 2 \cdot \cfrac{d^2}{(d+1)^2} \cdot |O\E Y_O|^2.
    \end{align*}

    Between an interior and a boundary point --- we substitute $\beta = 1$ and integrate over $\alpha$:
    \begin{align*}
        \E|XY_O|^2
        &= \int_0^1 f(\alpha, 1) \cdot d \cdot \alpha^{d-1}\, d\alpha \\
        &= \int_0^1 \bigl[(\alpha^2 + 1) \cdot \E|OY_O|^2 - 2 \alpha \cdot |O\E Y_O|^2\bigr] \cdot d \cdot \alpha^{d-1}\, d\alpha \\
        &= \left(1 + \cfrac{d}{d+2}\right) \cdot \E|OY_O|^2 - 2 \cdot \cfrac{d}{d+1} \cdot |O\E Y_O|^2.
    \end{align*}

    Between two boundary points --- we substitute $\alpha = \beta = 1$:
    $$\E|Y_O Y_O'|^2 = f(1, 1) = 2 \cdot \E|OY_O|^2 - 2 \cdot |O\E Y_O|^2.$$
\end{proof}

Combining this theorem with the previous one, we obtain the desired result.
\begin{corollary}
    If $O$ is the centre of mass of the corresponding projective distribution, then
    $$\E|XX'|^2 = \cfrac{d}{d+1} \cdot \E|XY_O|^2 = \cfrac{2d}{d+2} \cdot \E|OY_O|^2 = \cfrac{d}{d+2} \cdot \E|Y_OY_O'|^2.$$
\end{corollary}

Such a point $O$ exists: the map $O \mapsto \E Y_O$ continuously maps the convex compact set $K$ into itself, so by Brouwer's theorem it has a fixed point.

\begin{corollary}
    If $O$ is the centre of the sphere inscribed in a centrally symmetric body, then
    $$\E|XX'|^2 = \cfrac{d}{d+1} \cdot \E|XY|^2 = \cfrac{2d}{d+2} \cdot \E|OY|^2 = \cfrac{d}{d+2} \cdot \E|YY'|^2.$$
\end{corollary}

\section{The case of a triangle}
\label{sec:triangles}

The triangle is the simplest and the most representative example of a non-symmetric convex body. It is all the more interesting that for it the Zaporozhets--Tarasov conjecture not only holds but also admits a strengthening: an inequality holds between the \emph{one-dimensional projections} of the distributions of the interior and the boundary points onto each direction separately. By the Cauchy--Crofton formula, such a per-projection inequality automatically yields the comparison of the full mean distances as well.

Let us state this strengthening as a separate conjecture.

\begin{hypothesis}[Per-projection conjecture]\label{hyp:projection}
For any convex body $K \subset \R^2$ and any direction $u \in [0, \pi)$
\[
    \E |P_u(X) - P_u(X')| \;\leq\; \E |P_u(Y) - P_u(Y')|,
\]
where $X, X'$ are independent uniform points inside $K$, $Y, Y'$ are independent uniform points on $\partial K$, and $P_u$ is the orthogonal projection onto the direction $u$.
\end{hypothesis}

Applying the Cauchy--Crofton formula (see Section~\ref{sec:centrally-symmetric}) to both sides and integrating over $u$, we obtain $\Delta(K) \leq \theta(K)$. Thereby the per-projection conjecture implies the original Zaporozhets--Tarasov conjecture. It is known that it holds for centrally symmetric bodies in $\R^2$ (this is, in essence, the content of Section~\ref{sec:centrally-symmetric}). The main goal of the present section is to prove it for all triangles as well.

\begin{theorem}\label{thm:triangle-projection}
For any triangle $K \subset \R^2$ and any direction $u \in [0, \pi)$ the per-projection inequality of Conjecture~\ref{hyp:projection} holds. In particular, for any triangle
\[
    \Delta(K) \;\leq\; \theta(K).
\]
\end{theorem}

The remaining part of the section is devoted to the proof of this theorem.

\subsection{Reduction to one-dimensional projections}

Fix a direction $u$. For independent one-dimensional random variables $\xi$, $\xi'$ with common distribution function $F$ the classical formula
\[
    \E |\xi - \xi'| \;=\; 2 \int_{-\infty}^{+\infty} F(t)\bigl(1 - F(t)\bigr)\, dt
\]
holds. Therefore the inequality of Conjecture~\ref{hyp:projection} for a fixed direction $u$ is equivalent to the comparison of two integrals of products $F(1-F)$, where $F$ is the distribution function of the projection of the interior (respectively, the boundary) point onto $u$.

Let us choose the coordinate system so that $u$ coincides with the abscissa axis. By a shift and a scaling we may assume that $P_u(K) = [0, 1]$. The vertices of the triangle are projected into three points $0, p, 1$ for some $p \in (0, 1)$ (after ordering), so that the vertices have the form $V_1 = (0, y_a)$, $V_2 = (p, y_b)$, $V_3 = (1, y_c)$.

\subsection{The distribution of the projection of an interior point}

The width function of the triangle $w(x)$ --- the distance between the upper and the lower boundary at a fixed abscissa $x$ --- is piecewise linear with $w(0) = w(1) = 0$ and a unique break point at $x = p$. The density of the projection of an interior point is proportional to $w$:
\[
    f_X(x) \;=\; \frac{w(x)}{S}, \qquad S \;=\; \int_0^1 w(x)\, dx \;=\; \frac{h}{2}, \qquad h \;=\; w(p).
\]
Direct integration gives the explicit form of the distribution function:
\[
F_X(x) \;=\;
\begin{cases}
\dfrac{x^2}{p}, & x \in [0, p], \\[6pt]
1 - \dfrac{(1-x)^2}{1-p}, & x \in [p, 1].
\end{cases}
\]
It is essential that \emph{$F_X$ depends only on the parameter $p$} and depends neither on the height $h$ nor on the particular ordinates of the vertices.

\subsection{The distribution of the projection of a boundary point}

The boundary of the triangle consists of three edges. Each of them is a segment of a straight line, onto whose projection the uniform distribution is projected uniformly. We denote the lengths of the edges by
\[
    \ell_1 \;=\; \sqrt{p^2 + (y_b - y_a)^2}, \qquad
    \ell_2 \;=\; \sqrt{(1-p)^2 + (y_c - y_b)^2}, \qquad
    \ell_3 \;=\; \sqrt{1 + (y_c - y_a)^2},
\]
where $\ell_1$ is the edge $V_1 V_2$ (projected uniformly onto $[0, p]$), $\ell_2$ is the edge $V_2 V_3$ (uniformly onto $[p, 1]$), $\ell_3$ is the edge $V_1 V_3$ (uniformly onto $[0, 1]$). We denote the perimeter by $P = \ell_1 + \ell_2 + \ell_3$. The distribution function of the projection of a boundary point is piecewise linear:
\[
F_Y(x) \;=\;
\begin{cases}
\alpha\, x, & x \in [0, p], \\
A + \beta(x - p), & x \in [p, 1],
\end{cases}
\]
where
\[
\alpha = \frac{\ell_1 + p\, \ell_3}{p\, P}, \qquad
\beta = \frac{\ell_2 + (1-p) \ell_3}{(1-p)\, P}, \qquad
A = F_Y(p) = \frac{\ell_1 + p\, \ell_3}{P}.
\]
The parameter $A$ is the fraction of the mass of the boundary projected into $[0, p]$; it is precisely this parameter that will enter the final quadratic inequality.

\subsection{Computation of the integrals}

\paragraph*{Interior points.} On the segment $[0, p]$ we have $F_X = t^2/p$, and
\[
    \int_0^p \frac{t^2}{p}\left(1 - \frac{t^2}{p}\right) dt \;=\; \frac{p^2(5 - 3p)}{15}.
\]
Analogously on $[p, 1]$:
\[
    \int_p^1 F_X (1 - F_X)\, dt \;=\; \frac{(1-p)^2 \bigl(5 - 3(1-p)\bigr)}{15}.
\]
Adding the two summands, we obtain
\[
    I(p) \;:=\; \int_0^1 F_X(t)\bigl(1 - F_X(t)\bigr)\, dt \;=\; \frac{p^2 - p + 2}{15}.
\]

\paragraph*{Boundary points.} Put $\bar A = 1 - A$. On $[0, p]$ the function $F_Y = (A/p)\, t$, therefore
\[
    \int_0^p \frac{A}{p} t \left(1 - \frac{A}{p} t\right) dt \;=\; \frac{A p}{2} - \frac{A^2 p}{3} \;=\; \frac{A p (3 - 2 A)}{6}.
\]
On $[p, 1]$, with the substitution $u = t - p$ and $F_Y = A + \frac{\bar A}{1-p}\, u$:
\[
    \int_0^{1-p} \left(A + \frac{\bar A}{1-p} u\right)\!\!\left(\bar A - \frac{\bar A}{1-p} u\right) du
    \;=\; \frac{\bar A (1 + 2 A)(1 - p)}{6}.
\]
In total
\[
    G(A, p) \;:=\; \int_0^1 F_Y(t)\bigl(1 - F_Y(t)\bigr)\, dt \;=\; \frac{A p (3 - 2 A) + \bar A (1 - p)(1 + 2 A)}{6}.
\]

\subsection{Reduction to a quadratic inequality}

The required per-projection inequality is equivalent to $G(A, p) \geq I(p)$. After the substitution $\bar A = 1 - A$ the numerator of $G$ reduces to
\[
    -2 A^2 + A (2 p + 1) + (1 - p),
\]
and the comparison $G \geq I$, after multiplication by $30$, takes the form
\[
    f(A, p) \;:=\; 10 A^2 - (10 p + 5) A + (2 p^2 + 3 p - 1) \;\leq\; 0.
\]

Thus, Theorem~\ref{thm:triangle-projection} reduces to the following two lemmas.

\begin{lemma}[Range of the parameter $A$]\label{lem:triangle-A-range}
For any triangle and any direction of projection
\[
    \frac{p}{2} \;\leq\; A \;\leq\; \frac{1 + p}{2}.
\]
\end{lemma}

\begin{proof}
The lengths $\ell_1, \ell_2, \ell_3$ satisfy the triangle inequality, and each edge is no shorter than its projection: $\ell_1 \geq p$, $\ell_2 \geq 1 - p$, $\ell_3 \geq 1$.

\smallskip
\emph{The lower bound.} The inequality $A \geq p/2$ is equivalent to $\ell_1 (2 - p) \geq p (\ell_2 - \ell_3)$. If $\ell_2 \leq \ell_3$, the right-hand side is non-positive while the left-hand side is positive. If, however, $\ell_2 > \ell_3$, then by the triangle inequality $\ell_2 - \ell_3 < \ell_1$, whence $p(\ell_2 - \ell_3) < p \ell_1 \leq (2 - p) \ell_1$ for $p \in (0, 1)$.

\smallskip
\emph{The upper bound.} The inequality $A \leq (1 + p)/2$ is equivalent to $(1 - p)(\ell_1 - \ell_3) \leq (1 + p) \ell_2$. If $\ell_1 \leq \ell_3$, the left-hand side is non-positive. If $\ell_1 > \ell_3$, then by the triangle inequality $\ell_1 - \ell_3 < \ell_2$, whence $(1 - p)(\ell_1 - \ell_3) < (1 - p) \ell_2 \leq (1 + p) \ell_2$.
\end{proof}

\begin{lemma}[Quadratic inequality]\label{lem:triangle-quadratic}
For all $p \in (0, 1)$ and $A \in [p/2, (1 + p)/2]$ one has $f(A, p) \leq 0$.
\end{lemma}

\begin{proof}
The function $f(\cdot, p)$ is quadratic in $A$ with a positive leading coefficient, so it suffices to verify its non-positivity at the endpoints of the segment. A direct computation gives
\[
    f\!\left(\tfrac{p}{2},\, p\right) \;=\; -\frac{p^2 - p + 2}{2}, \qquad
    f\!\left(\tfrac{1 + p}{2},\, p\right) \;=\; -\frac{p^2 - p + 2}{2}.
\]
Since $p^2 - p + 2 = (p - \tfrac{1}{2})^2 + \tfrac{7}{4} > 0$ for all real $p$, both values are negative. Consequently, $f(A, p) \leq -\tfrac{1}{2}(p^2 - p + 2) < 0$ on the whole segment.
\end{proof}

Lemmas~\ref{lem:triangle-A-range} and~\ref{lem:triangle-quadratic} together prove $G(A, p) \geq I(p)$ for all admissible pairs $(A, p)$, which completes the proof of Theorem~\ref{thm:triangle-projection}.

% \begin{note_nnum}
% Only two properties of the triangle were used in the proof: (a) the width function is piecewise linear with a unique break point --- this gives a simple explicit form of $F_X$; (b) the projections of the edges onto a line are segments --- this ensures the piecewise linearity of $F_Y$. The triangle inequality for the lengths of the sides plays a key role at the stage of estimating the parameter $A$.
% \end{note_nnum}

\begin{note_nnum}
The proof gives a quantitative gap: for a triangle $K$
\[
    \theta(K) - \Delta(K) \;\geq\; c_1 \int_0^{\pi} w(u)\,\frac{p(u)^2 - p(u) + 2}{30}\, du \;>\; 0,
\]
where $c_1 = 1/2$ is the Cauchy--Crofton constant for $\rho = 1$, $w(u)$ is the width of $K$ in the direction $u$, and $p(u) \in (0, 1)$ is the parameter determining the projection of the vertices onto the direction $u$. Since $p^2 - p + 2 \geq 7/4$ for all $p$, and $\int_0^{\pi} w(u)\, du = |\sigma K|$ (Cauchy's formula), this yields the rough estimate $\theta(K) \geq \Delta(K) + \tfrac{7}{240}\, |\sigma K|$.
\end{note_nnum}

\section{Other approaches to the problem}

In this section we would like to present some other approaches to the problem which, although they have not yet led to substantive general results, seem to us sufficiently interesting and worth mentioning.

\subsection{On the necessity of studying the distance between an interior and a boundary point}

The main essence of the problem is to compare the distributions of the distances between two random interior points $|XX'|$ and two random boundary points $|YY'|$, which turns out to be a rather difficult problem. One has to study separately the properties of the interior and of the boundary distributions, which are then quite hard to relate to each other.

This is precisely why in the section on exact relations we began to consider the distance between an interior and a boundary point $|XY|$. It turns out that studying the relation of this distance with any of those introduced earlier is much easier, since half of the definition already coincides.

Although exact relations can be obtained only for a narrow class of bodies, we can show in general form that in order to compare the interior and the boundary distances it suffices to compare the mean mixed distance with one of them.

\begin{statement}\label{stmt:cnd-inequality}
    Let a random variable $X$ have the distribution $\distr{X}$, and let a variable $Y$ independent of it have the distribution $\distr{Y}$. Then for any $\rho \in (0, 2]$
    $$\E|XY|^\rho \;\geq\; \cfrac{\E|XX'|^\rho + \E|YY'|^\rho}{2}.$$
\end{statement}

From this there follows

\begin{corollary}
    Under the same hypotheses, for any $\rho \in (0, 2]$
    $$\E|XY|^\rho \;\leq\; \E|YY'|^\rho \;\so\; \E|XX'|^\rho \;\leq\; \E|YY'|^\rho.$$
\end{corollary}

The technical proof and the explicit expressions for the defects in the cases $\rho = 1$ and $\rho = 2$ will be given below. First let us discuss the applications of these inequalities.

Consider the function $f(x) := \E|x - Y|^p$. Obviously, $\E|XY|^p = \E f(X)$ and $\E|YY'|^p = \E f(Y)$. Note that $f$ is convex for $p \geq 1$ --- as the expectation of the function $x \mapsto |x - y|^p$, which is convex in $x$. Thereby, in order to verify Conjecture~\ref{hyp:zt} it suffices to establish that the mean value of this convex function with respect to the uniform distribution inside a convex body does not exceed its mean value on the boundary of the same body.

In $\R^1$ this statement is indeed true for any convex function; however, in $\R^2$ it already fails. A sufficiently long and narrow isosceles triangle can serve as a counterexample, with the square of the distance to its apex as the convex function. The point is that the part of the mass of the interior located near the base of the triangle exceeds the corresponding part of the mass of the perimeter, since the width of the triangle increases as one moves from the apex to the base.

Nevertheless, this point of view allows one to establish, almost without effort, the validity of the Zaporozhets--Tarasov conjecture for centrally symmetric circumscribed bodies in any dimension, and also opens the way to considering the case of parallelepipeds in an arbitrary dimension, which did not seem possible before.

\begin{statement}
    Let $K \subset \R^d$ be a convex centrally symmetric circumscribed body. Then
    $$\forall p \in [1, 2] \quad \Delta^p(K) < \theta^p(K).$$
\end{statement}

\begin{proof}
By what has been said above it suffices to establish $\E f(X) \leq \E f(Y)$, where $f(x) = \E|x - Y|^p$ is a convex function (the strict inequality $\Delta^p(K) < \theta^p(K)$ then follows from the corollary to Proposition~\ref{stmt:cnd-inequality}, since the distributions of $X$ and $Y$ are different).

Let us use the representation of the points of a circumscribed body in homothetic coordinates with centre at the centre of the inscribed sphere (see Section~\ref{sec:homothetic}): $X = \alpha\, \gamma(s)$, where $\alpha \in [0, 1]$ is the radial coordinate with density $d\, \alpha^{d-1}$, $s \in S^{d-1}$ is the direction with distribution $\mu_S$, and $\alpha$ and $s$ are independent. For a centrally symmetric circumscribed body the boundary point $Y$ has the form $\gamma(s)$ with the same distribution of the direction.

Let us split the directions into pairs of opposite ones $\{s, -s\}$. By the central symmetry of the body $\gamma(-s) = -\gamma(s)$, and the distribution $\mu_S$ is invariant under $s \mapsto -s$. This allows us to rewrite both integrals in symmetric form:
\[
    \E f(X) \;=\; \E_{(\alpha, s)}\, \frac{f\bigl(\alpha\, \gamma(s)\bigr) + f\bigl(-\alpha\, \gamma(s)\bigr)}{2}, \qquad
    \E f(Y) \;=\; \E_s\, \frac{f\bigl(\gamma(s)\bigr) + f\bigl(-\gamma(s)\bigr)}{2}.
\]
Therefore it suffices to establish that for all $\alpha \in [0, 1]$ and any vector $v \in \R^d$
\[
    \frac{f(\alpha\, v) + f(-\alpha\, v)}{2} \;\leq\; \frac{f(v) + f(-v)}{2}.
\]

This statement --- that the mean value of a convex function on a symmetric pair $\{\alpha v,\, -\alpha v\}$ does not exceed its mean on the pair of endpoints of the segment $\{v, -v\}$ --- is a direct consequence of convexity. By the convexity of $f$ the function $\alpha \mapsto f(\alpha v) + f(-\alpha v)$ is non-decreasing in $\alpha$ on $[0, 1]$: as $\alpha$ increases, the symmetric pair $\{\alpha v, -\alpha v\}$ ``spreads apart'' from the centre of the segment $[-v, v]$ while the arithmetic mean is preserved, and the sum of the values of a convex function on such a pair only grows from this. In particular, it does not exceed its value at $\alpha = 1$, which is what was required.
\end{proof}

\begin{statement}
    Let $K \subset \R^d$ be a parallelepiped formed by the vectors $v_1, \ldots, v_d$. Then
    $$\forall p \in [1, 2] \quad \Delta^p(K) < \theta^p(K).$$
\end{statement}

\begin{proof}
    As was noted above, it suffices to show that $\E f(X) \leq \E f(Y)$. Let us split the faces into pairs of parallel ones $F_{i,1}, F_{i,2}$, differing by a translation by the vector $v_i$, so that $F_{i,1} + v_i = F_{i,2}$. By the law of total probability:
    \begin{align*}
        \E f(Y)
        &= \sum\limits_{i=1}^{d} \frac{2|F_{i,1}|}{|\partial K|} \cdot \E \bigl[f(Y) \,\big|\, Y \in F_{i,1} \cup F_{i,2}\bigr] \\
        &= \sum\limits_{i=1}^{d} \frac{2|F_{i,1}|}{|\partial K|} \cdot \E \!\left[\frac{f(W) + f(W + v_i)}{2} \,\Big|\, W \in F_{i,1}\right] \\
        &\geq \sum\limits_{i=1}^{d} \frac{2|F_{i,1}|}{|\partial K|} \cdot \E \!\left[ \int_{0}^{1} f(W + t v_i)\, dt \,\Big|\, W \in F_{i,1}\right] \\
        &= \sum\limits_{i=1}^{d} \frac{2|F_{i,1}|}{|\partial K|} \cdot \E f(X) \;=\; \E f(X).
    \end{align*}
    The inequality is a consequence of the convexity of $f$ on the segment $[w, w + v_i]$. The penultimate equality holds because the map $(w, t) \mapsto w + t v_i$ defines an affine bijection $F_{i,1} \times [0, 1] \to K$ with constant Jacobian and, consequently, carries the uniform distribution on $F_{i,1} \times [0, 1]$ into the uniform distribution on $K$; therefore $X \overset{d}{=} W + T v_i$, where $W$ is uniform on $F_{i,1}$, $T$ is uniform on $[0, 1]$, and they are independent, whence $\E f(X) = \E\bigl[ \int_{0}^{1} f(W + t v_i)\, dt \,\big|\, W \in F_{i,1} \bigr]$. The last equality follows from $\sum_{i=1}^{d} 2|F_{i,1}| = |\partial K|$. This implies the stated assertion.
\end{proof}

Let us pass to the proof of Proposition~\ref{stmt:cnd-inequality}. The statement itself is a classical result: it is equivalent to the conditional negative definiteness of the kernel $|x - y|^\rho$ for $\rho \in (0, 2]$, established by Schoenberg~\cite[\textsection\,3, p.~525 and \textsection\,6, Theorem~5]{Schoenberg1938} (the embeddability of the spaces $E_m^p(\gamma)$ into a Hilbert space for $\gamma \leq p/2$, whence $|x-y|^\rho$ turns out to be conditionally negative definite). The probabilistic form of the statement, in which the finite sums $\sum \rho_i \rho_k\, |x_i - x_k|^\rho$ are replaced by double integrals with respect to the signed measure $\nu = \distr{X} - \distr{Y}$ of total mass zero, is written out explicitly, for example, in~\cite{BergChristensenRessel1984}. The case of $\R^d$ is obtained from the one-dimensional one by applying the Cauchy--Crofton formula: the one-dimensional inequality applied to the projections $P_u(X), P_u(Y)$ onto each direction $u$ gives, after integration over $u$, the $d$-dimensional version.

Let us give a proof based on the Lévy representation for the function $|z|^\rho$. First of all let us introduce an auxiliary notion which will be useful to us more than once in what follows.

\begin{definition}[Variance of a point]\label{def:variance-point}
Let $X$ be a random point in $\R^d$ with finite second moment. Its \emph{variance} is defined to be
\[
    \D X \;:=\; \E |X - \E X|^2 \;=\; \sum\limits_{i = 1}^{d} \D X_i,
\]
that is, the sum of the variances of the coordinates of the point.
\end{definition}

\begin{statement}[The mean square of the distance via the variances]\label{stmt:distance-variance}
For any independent random points $X$, $Y$ in $\R^d$ with finite second moments one has
\[
    \E |X - Y|^2 \;=\; \D X \;+\; \D Y \;+\; |\E X - \E Y|^2.
\]
In particular, for two independent copies of one point $\E |X - X'|^2 = 2 \D X$.
\end{statement}

\begin{proof}
Let us write $X - Y = (X - \E X) - (Y - \E Y) + (\E X - \E Y)$ and expand the square of the Euclidean norm. Besides the squares of the three summands, three cross products appear, and all of them vanish:
\begin{itemize}
    \item $\E \langle X - \E X,\; Y - \E Y \rangle = \langle \E(X - \E X),\, \E(Y - \E Y) \rangle = 0$ --- by independence and centring;
    \item $\E \langle X - \E X,\; \E X - \E Y \rangle = \langle \E(X - \E X),\, \E X - \E Y \rangle = 0$ --- by centring;
    \item $\E \langle Y - \E Y,\; \E X - \E Y \rangle = 0$ --- analogously.
\end{itemize}
There remain three squares:
\[
    \E |X - Y|^2 \;=\; \E |X - \E X|^2 + \E |Y - \E Y|^2 + |\E X - \E Y|^2 \;=\; \D X + \D Y + |\E X - \E Y|^2.
\]
Applying this to $Y = X'$ (an independent copy): $\E X' = \E X$ annihilates the last summand, and $\E |X - X'|^2 = 2 \D X$.
\end{proof}

\begin{proof}[Proof of Proposition~\ref{stmt:cnd-inequality}]
Let us introduce the signed measure $\nu := \distr{X} - \distr{Y}$ of total mass $\nu(\R) = 0$. A direct expansion gives
\[
    \iint |x - y|^\rho \, d\nu(x)\, d\nu(y)
    \;=\; \E|X - X'|^\rho - 2\, \E|X - Y|^\rho + \E|Y - Y'|^\rho.
\]
Therefore the inequality of the proposition is equivalent to
\begin{equation}\label{eq:cnd-form}
    \iint |x - y|^\rho \, d\nu(x)\, d\nu(y) \;\leq\; 0.
\end{equation}

\smallskip
\noindent\emph{The case $\rho \in (0, 2)$.} Let us use the Lévy representation: for $\rho \in (0, 2)$ there is a constant $c_\rho > 0$ such that
\begin{equation}\label{eq:levy}
    |z|^\rho \;=\; c_\rho \int_0^\infty \frac{1 - \cos(tz)}{t^{1+\rho}}\, dt
    \qquad \forall z \in \R.
\end{equation}
By the substitution $u = t|z|$ and the evenness of the cosine, the right-hand side equals $|z|^\rho \cdot I_\rho$, where
$$I_\rho := \int_0^\infty \frac{1 - \cos u}{u^{1+\rho}}\, du.$$
This integral converges for $\rho \in (0, 2)$: at zero $1 - \cos u \sim u^2/2$, so the integrand behaves like $u^{1-\rho}$ (integrable for $\rho < 2$), and at infinity $|1 - \cos u| \leq 2$ and $u^{-1-\rho}$ is integrable for $\rho > 0$. Put $c_\rho := 1/I_\rho$.

Substituting~\eqref{eq:levy} into~\eqref{eq:cnd-form} and using $\nu(\R) = 0$:
\begin{align*}
    \iint |x - y|^\rho \, d\nu(x)\, d\nu(y)
    &= c_\rho \int_0^\infty \frac{1}{t^{1+\rho}}
       \iint \bigl(1 - \cos(t(x - y))\bigr)\, d\nu(x)\, d\nu(y)\, dt \\
    &= -c_\rho \int_0^\infty \frac{1}{t^{1+\rho}}
       \iint \cos(t(x - y))\, d\nu(x)\, d\nu(y)\, dt.
\end{align*}
Expanding $\cos(t(x - y)) = \cos(tx)\cos(ty) + \sin(tx)\sin(ty)$:
$$\iint \cos(t(x - y))\, d\nu(x)\, d\nu(y)
    \;=\; \Bigl(\!\int \cos(tx)\, d\nu(x)\Bigr)^2
       + \Bigl(\!\int \sin(tx)\, d\nu(x)\Bigr)^2
    \;=\; |\hat\nu(t)|^2 \;\geq\; 0,$$
where $\hat\nu(t) := \int e^{itx}\, d\nu(x)$ is the characteristic function of the measure $\nu$. Finally
\begin{equation}\label{eq:cnd-final}
    \iint |x - y|^\rho \, d\nu(x)\, d\nu(y)
    \;=\; -c_\rho \int_0^\infty \frac{|\hat\nu(t)|^2}{t^{1+\rho}}\, dt
    \;\leq\; 0.
\end{equation}

\smallskip
\noindent\emph{The case $\rho = 2$.} The integral~\eqref{eq:levy} diverges at zero, and the Lévy formula does not work directly. On the other hand, everything can be computed explicitly via Proposition~\ref{stmt:distance-variance}: substituting
\[
    \E |X - X'|^2 = 2 \D X, \qquad
    \E |Y - Y'|^2 = 2 \D Y, \qquad
    \E |X - Y|^2 = \D X + \D Y + |\E X - \E Y|^2,
\]
we have
\begin{align*}
    \E |X - Y|^2 - \frac{\E |X - X'|^2 + \E |Y - Y'|^2}{2}
    &= \bigl(\D X + \D Y + |\E X - \E Y|^2\bigr) - \bigl(\D X + \D Y\bigr) \\
    &= |\E X - \E Y|^2 \;\geq\; 0.
\end{align*}
\end{proof}

\begin{note_nnum}[The defect in the case $\rho = 2$]
The proof for $\rho = 2$ gives an explicit expression for the defect of the inequality of Proposition~\ref{stmt:cnd-inequality}:
\[
    \E |X - Y|^2 \;-\; \frac{\E |X - X'|^2 + \E |Y - Y'|^2}{2}
    \;=\; |\E X - \E Y|^2.
\]
In other words, the defect for $\rho = 2$ is exactly the square of the distance between the expectations (the centres) of the distributions of $X$ and $Y$. In particular, equality is attained if and only if $\E X = \E Y$.
\end{note_nnum}

\begin{note_nnum}[The defect for $\rho = 1$ in the one-dimensional case]
In the one-dimensional case ($d = 1$) for $\rho = 1$ the defect also admits an explicit expression --- through the distribution functions. The length of the segment $|X - Y|$ equals the measure of those points $t \in \R$ that lie between $X$ and $Y$:
\[
    |X - Y| \;=\; \int_\R \Ind\{t \text{ lies between } X \text{ and } Y\}\, dt.
\]
A point $t$ lies between $X$ and $Y$ if and only if one of them does not exceed $t$ and the other is not less than $t$. By independence
\[
    \P\bigl(t \text{ lies between } X \text{ and } Y\bigr) \;=\; F_X(t)\bigl(1 - F_Y(t)\bigr) + F_Y(t)\bigl(1 - F_X(t)\bigr),
\]
where $F_X, F_Y$ are the distribution functions. Taking the expectation of the length and applying Fubini's theorem, we obtain
\[
    \E|X - Y| \;=\; \int_\R \bigl[F_X(1 - F_Y) + F_Y(1 - F_X)\bigr]\, dt.
\]
In the particular cases of independent copies ($Y = X'$, respectively $X = Y'$):
\[
    \E|X - X'| \;=\; 2 \int_\R F_X(1 - F_X)\, dt, \qquad
    \E|Y - Y'| \;=\; 2 \int_\R F_Y(1 - F_Y)\, dt.
\]

Let us substitute this into the defect:
\begin{align*}
    \E|X - Y| - \frac{\E|X - X'| + \E|Y - Y'|}{2}
    &= \int_\R \bigl[F_X(1 - F_Y) + F_Y(1 - F_X) \\
    &\qquad\qquad\quad - F_X(1 - F_X) - F_Y(1 - F_Y)\bigr]\, dt.
\end{align*}
Expanding the brackets in the integrand:
\begin{align*}
    & F_X(1 - F_Y) + F_Y(1 - F_X) - F_X(1 - F_X) - F_Y(1 - F_Y) \\
    &\qquad =\; (F_X - F_X F_Y) + (F_Y - F_X F_Y) - (F_X - F_X^2) - (F_Y - F_Y^2) \\
    &\qquad =\; F_X^2 - 2 F_X F_Y + F_Y^2 \\
    &\qquad =\; (F_X - F_Y)^2.
\end{align*}
Finally
\[
    \E|X - Y| \;-\; \frac{\E|X - X'| + \E|Y - Y'|}{2}
    \;=\; \int_\R \bigl(F_X(t) - F_Y(t)\bigr)^2\, dt.
\]
That is, the defect for $\rho = 1$ is the square of the $L^2$-distance between the distribution functions. In particular, the inequality is strict whenever $F_X \not\equiv F_Y$.
\end{note_nnum}

\subsection{Ohlin's lemma and its applications}

In the course of studying the question we managed to discover an extremely entertaining classical statement --- Ohlin's lemma~\cite{Ohlin1969}.

\begin{lemma_nnum}[Ohlin]\label{lem:ohlin}
Let $X, Y$ be one-dimensional random variables with finite and equal means $\E X = \E Y$ and distribution functions $F_X, F_Y$. If there is a point $t_0 \in \R$ such that
\[
    F_X(t) \leq F_Y(t) \text{ for } t < t_0, \qquad
    F_X(t) \geq F_Y(t) \text{ for } t > t_0,
\]
then for any continuous convex function $\varphi$ \textup{(}with finite $\E\varphi(X), \E\varphi(Y)$\textup{)} one has
\[
    \E\varphi(X) \;\leq\; \E\varphi(Y).
\]
\end{lemma_nnum}

At first sight the statement is counter-intuitive --- the requirement on the variables themselves is too weak for such a strong result. Nevertheless the proof relies on an extremely simple idea: integration by parts and supporting the convex function by a supporting line at the point $t_0$ (see~\cite{Ohlin1969}). This technique deserves an attempt at transfer.

Let us apply it to obtain a weakened version of the Zaporozhets--Tarasov conjecture. Note that by the same technique that was used in the centrally symmetric case (Section~\ref{sec:centrally-symmetric}) one can show that the distribution functions of the projections $F_{X_u}, F_{Y_u}$ for planar bodies do indeed always have a unique intersection for each direction.

Indeed, the proof of the increase of the ratio ``area to perimeter'' for the part of the body lying to the left of the scanning line essentially used only the fact that the tangents to the boundary at the points of its intersection with the scanning line are located on the required side of the line. This means that the ratio increases up to the moment when these tangents become parallel to each other. A symmetric analysis shows that when scanning from the opposite end the analogous ratio also increases as one moves away from the same limiting position.

It remains to note that the break point --- the one at which the tangents are parallel --- either itself falls on the quantile $1/2$, or the values up to which these ratios increase turn out to be not less than one on exactly one side. In both cases $F_{X_u}$ and $F_{Y_u}$ can intersect at most once. If, however, there is no suitable break point inside the segment, one can take its endpoint as such.

Thus, for the distributions of interest to us the one-dimensional projections $X_u, Y_u$ have distribution functions intersecting at exactly one point for each direction $u \in S^{d-1}$, and the hypothesis of the following theorem is satisfied automatically.

\begin{theorem}\label{thm:ohlin-application}
Let $X, Y$ be two independent random points in $\R^d$, and suppose that for each direction $u \in S^{d-1}$ the distribution functions of the projections $F_{X_u}, F_{Y_u}$ intersect exactly once at a point $t_0(u)$, with $F_{X_u}(t_0(u)) = F_{Y_u}(t_0(u)) =: p(u)$ and $F_{X_u}(t) \leq F_{Y_u}(t)$ for $t<t_0(u)$ and $F_{X_u}(t) \geq F_{Y_u}(t)$ for $t>t_0(u)$. Then
\[
    \E|X - X'| \;-\; \E|Y - Y'|
    \;\leq\; 2 c_d \int_{S^{d-1}} (\E X_u - \E Y_u) \cdot (2 p(u) - 1)\, du,
\]
where $c_d > 0$ is the Cauchy--Crofton constant. In particular,
\[
    \bigl|\E|X - X'| - \E|Y - Y'|\bigr| \;\leq\; 2 \,|\E X - \E Y|.
\]
\end{theorem}

\begin{proof}
Let us reduce the problem to a one-dimensional one by the Cauchy--Crofton formula: $\E|X - Y| = c_d \int_{S^{d-1}} \E|X_u - Y_u|\, du$ and analogously for the remaining pairs. It suffices to establish the analogous relation for the projections for each fixed direction $u$. In what follows we omit the index $u$ and work in $\R^1$.

\medskip
Instead of estimating the difference $\E|X-X'| - \E|Y-Y'|$ directly, let us insert between the two ``homogeneous'' pairs the mixed pair $X, Y$:
\[
    \E|X - X'| - \E|Y - Y'|
    \;=\; \bigl(\E|X-X'| - \E|X-Y|\bigr) \;+\; \bigl(\E|X-Y| - \E|Y-Y'|\bigr).
\]
Each of the two summands has the form $\E\varphi(X) - \E\varphi(Y)$ for a suitable convex function $\varphi$, to which the technique of Ohlin's lemma is applicable.

\medskip
\noindent\emph{The first summand.} Let us introduce $g(z) := \E|X - z|$, a convex function (as the expectation of the convex $z \mapsto |x - z|$). Then $\E|X - X'| = \E g(X)$, $\E|X - Y| = \E g(Y)$, and
\[
    \E|X - X'| - \E|X - Y|
    \;=\; \int_\R g(z)\, d\bigl[F_X - F_Y\bigr](z).
\]
Let us support $g$ by the supporting line $\ell(z) = a(z - t_0) + g(t_0)$ at the point of intersection $t_0$, where $a = g'(t_0) = 2 F_X(t_0) - 1 = 2 p - 1$. Then $g(z) = \ell(z) + \phi(z)$, where $\phi \geq 0$ is a convex function with minimum $\phi(t_0) = 0$. Substituting:
\[
    \int g\, d[F_X - F_Y]
    \;=\; \underbrace{a \cdot \int (z - t_0)\, d[F_X - F_Y]}_{= a (\E X - \E Y)}
    \;+\; \underbrace{g(t_0) \cdot \int d[F_X - F_Y]}_{= 0}
    \;+\; \int \phi\, d[F_X - F_Y].
\]
The non-linear remainder $\int \phi\, d[F_X - F_Y] \leq 0$. Indeed, by construction $\phi'(z) \leq 0$ for $z < t_0$ and $\phi'(z) \geq 0$ for $z > t_0$, while the condition of a single intersection gives $F_X(z) - F_Y(z) \leq 0$ for $z < t_0$ and $\geq 0$ for $z > t_0$. Therefore the integrand in $-\int \phi'(z)\,[F_X - F_Y](z)\, dz$ is everywhere non-negative, whence $\int \phi\, d[F_X - F_Y] = -\int \phi'(F_X - F_Y)\, dz \leq 0$.

As a result,
\[
    \E|X - X'| - \E|X - Y| \;\leq\; (2 p - 1)(\E X - \E Y).
\]

\medskip
\noindent\emph{The second summand.} With the function $\tilde g(z) := \E|Y - z|$ everything is analogous: $\tilde g'(z) = 2 F_Y(z) - 1$, and at the point of intersection $\tilde g'(t_0) = 2 F_Y(t_0) - 1 = 2 p - 1$ (the same value of the coefficient, since $F_X(t_0) = F_Y(t_0)$). The same support by a supporting line gives
\[
    \E|X - Y| - \E|Y - Y'| \;\leq\; (2 p - 1)(\E X - \E Y).
\]

\medskip
\noindent\emph{The sum.} Adding the two summands, we obtain for the one-dimensional projections
\[
    \E|X - X'| - \E|Y - Y'| \;\leq\; 2(2 p - 1)(\E X - \E Y).
\]
Returning to $\R^d$ by integration over the directions with the Cauchy--Crofton constant, we obtain the statement of the theorem.

The rough estimate $|2 p - 1| \leq 1$ together with the Cauchy--Crofton formula for $|\E X - \E Y| = c_d \int_{S^{d-1}} |\E X_u - \E Y_u|\, du$ gives the second statement of the theorem.
\end{proof}

\begin{corollary}
    For any convex body $K \subset \R^2$ such that $\E X = \E Y$, one has
    $$\Delta(K) \leq \theta(K).$$
\end{corollary}

In particular, using this statement, one can prove the result for planar centrally symmetric bodies and for an arbitrary moment; however, this requires a much heavier technique and conceals the essence of the stochastic majorisation lying behind it.

\subsection{Crofton lines}

Another classical approach to working with random interior points is that of \emph{Crofton lines}. On the space of all lines of the plane there exists a unique (up to a constant factor) measure invariant under motions. It is most conveniently described as follows: each line is specified by a pair consisting of a direction and a signed distance from the origin, and the invariant measure on this space is the direct product of the uniform measure on the circle of directions and the Lebesgue measure on the distances.

The set of lines meeting a given convex body $K$ has finite invariant measure; normalising it to unity, we obtain a probability space. The random lines from this space are called the \emph{Crofton lines} of the body~$K$; their intersections with the body define random chords. Already in 1969 Kingman~\cite{Kingman1969} established a fundamental relation between the moments of the distances between two random interior points of the body and the moments of the lengths of these chords:
\[
    \E|XX'|^p \;=\; \frac{d\, \kappa_d}{(d+p)(d+p+1)}\, I_{p+d+1},
\]
where $I_q$ is the $q$-th moment of the length of a random Crofton chord and $\kappa_d$ is the volume of the unit ball in $\R^d$. Moreover, the endpoints of a random chord, although dependent on each other, are individually distributed uniformly on the boundary of the body.

Unfortunately, a random point chosen uniformly along a random chord does not give a uniform distribution inside the body. On the other hand, the point of intersection of two \emph{independent} Crofton lines, conditioned on lying inside the body, is uniformly distributed in it. The discrepancy that arises is not hard to explain: when a random chord is chosen, all the chords are taken ``with equal weight'', whereas a genuine uniform interior point, by Fubini's theorem, falls on long chords more often than on short ones. This correction can be introduced artificially --- and then the problem admits the following reformulation.

\begin{statement}
Let $K \subset \R^2$ be a convex body, let $X$ be a random point uniformly distributed inside $K$, let $Y, Y'$ be independent random points uniformly distributed on $\partial K$, and let $l$ be a random Crofton line. Put
\[
    f(x) := \E|x - Y'|, \qquad g(x) := \E\bigl[|l \cap K| \,\big|\, x \in l\bigr].
\]
If $\E\bigl[f(Y)\, g(Y)\bigr] \leq \E f(Y) \cdot \E g(Y)$, then $\Delta(K) \leq \theta(K)$.
\end{statement}

\begin{proof}
Denote by $h_K(u) = \sup\limits_{x \in K} \langle x, u\rangle$ the support function of the body (a standard object): it shows how far the body extends in the direction $u$. The area of the body can be ``unfolded'' into an integral over the Crofton lines: for any direction $u \in [0, \pi]$ the chords $l(u, t) = K \cap \{\langle x, u\rangle = t\}$ for $t \in [h_K(-u), h_K(u)]$ sweep out the whole body, and integration over $u$ simply gives a $\pi$-fold repetition. In particular,
\[
    \E|XY| \;=\; \frac{1}{|K|} \int_K f(x)\, dx
    \;=\; \frac{1}{\pi \cdot |K|} \int_0^{\pi}\!\! du \int_{h_K(-u)}^{h_K(u)}\!\! dt \int_{l(u, t)} f(x)\, dx.
\]

Let us use the convexity of $f$: on any chord its mean does not exceed the mean of the values at the endpoints. If $a, b \in \partial K$ are the endpoints of the chord $l$, then
\[
    \int_{l} f(x)\, dx \;\leq\; |l| \cdot \frac{f(a) + f(b)}{2}.
\]
Substituting into the integral above and using the fact that the marginal distributions of the endpoints of a Crofton line coincide with the uniform measure on $\partial K$, we obtain
\[
    \E|XY| \;\leq\; \frac{|\partial K|}{\pi \cdot |K|} \cdot \E\bigl[f(Y)\, g(Y)\bigr].
\]
In this expression $g(Y)$ is the conditional mean of the length of a Crofton line given that one of its endpoints is $Y$.
Under the assumption from the statement of the proposition, $\E[f(Y) g(Y)] \leq \E f(Y) \cdot \E g(Y)$, therefore
\[
    \E|XY| \;\leq\; \frac{|\partial K|}{\pi \cdot |K|} \cdot \E f(Y) \cdot \E g(Y).
\]
By Cauchy's formula the mean length of a Crofton line passing through a uniformly chosen boundary point is exactly $\pi |K|/|\partial K|$, that is,
\[
    \E g(Y) \;=\; \frac{\pi \cdot |K|}{|\partial K|}.
\]
Substituting, we see that the outer factor $|\partial K|/(\pi |K|)$ cancels with $\E g(Y)$:
\[
    \E|XY| \;\leq\; \E f(Y) \;=\; \E|YY'|.
\]
It remains to apply the corollary to Proposition~\ref{stmt:cnd-inequality} with $\rho = 1$: from $\E|XY| \leq \E|YY'|$ it follows that $\E|XX'| \leq \E|YY'|$, that is, $\Delta(K) \leq \theta(K)$.
\end{proof}

Thereby, by and large, the Zaporozhets--Tarasov conjecture reduces to checking the sign of the correlation of the random variables $f(Y)$ and $g(Y)$ (up to an intermediate convexity estimate). However, because of the normalisation by the $\sin$ of the angle between the chord and the tangent, the function $g$ turns out to be hard to study, and we did not manage to obtain a formal derivation of the non-positivity of the covariance $\mathrm{Cov}(f(Y), g(Y))$ for an arbitrary convex body.

\section{Conclusion}

The Zaporozhets--Tarasov conjecture, posed in 2019, turned out to be an extremely interesting problem. In the course of the work we managed to achieve substantial progress. We have completely answered the conjecture in the centrally symmetric case: we confirmed it in dimension 2 and refuted it in all higher dimensions. We also managed to indicate explicitly a class of bodies for which the conjecture holds in all dimensions, and to verify the conjecture in the case of triangles.

Separate parts of the work were presented at several conferences of both all-Russian and international level. The centrally symmetric case is published in~\cite{lotnikov2025}; the counterexample is being prepared for publication as a preprint.

Nevertheless, we did not manage to answer the posed question completely: the general case of planar bodies still remains open. However, we have a number of further conjectures that deserve a detailed study; some of them are given in the present work, and we hope that one day this question will be given an unambiguous answer.

\section*{Acknowledgement}

Summing up the work, I cannot fail to express my gratitude to my academic advisor. Dmitry Nikolaevich supported and guided me in every way during the research, was always aware of all the details of the work being carried out and of the conjectures being explored. Many thanks for an extremely interesting problem and for the most interesting joint work.

\section*{Author's note added to the English translation (August 2026)}
\addcontentsline{toc}{section}{Author's note added to the English translation (August 2026)}

The Russian original of the present work was submitted as the author's Bachelor's thesis, entitled \emph{Mean distance between points inside and on the boundary of a convex body}, and was defended at Saint Petersburg State University on 15 June 2026. The present English text is a faithful translation of that thesis. No mathematical statement or proof from the Russian original has been altered.

A complementary result concerning the planar case was obtained by Maxim Kukushkin in the Russian-language manuscript \emph{The Zaporozhets--Tarasov Inequality for an Arbitrary Planar Convex Body}. A version of that manuscript containing the proof of the conjecture in dimension two was deposited in a private GitHub repository on 26 July 2026. The exact state of the manuscript deposited at that time is preserved in commit \texttt{17c5b853b52146f7960f286565a1492fb334424c}\footnote{\url{https://github.com/mkukushkin2004/mean-distance-inequality-proof/commit/17c5b853b52146f7960f286565a1492fb334424c}}. The repository was made public on 13 August 2026.

During the preparation of the present English translation, the author became aware of the independent preprint by Eric Shen, \emph{The Zaporozhets--Tarasov Conjecture on Mean Distances}, arXiv:2608.06470v1, first submitted on 6 August 2026. Shen's preprint contains both a proof of the conjecture in dimension two and counterexamples in every dimension \(d\geq 3\). The counterexamples established in the present thesis were obtained independently of Shen's work.

The present section, including the references to Kukushkin's original manuscript and to Shen's preprint, was added in August 2026 and was not part of the Bachelor's thesis defended in June 2026.

\end{document}